\documentclass[a4paper]{amsart}

\usepackage[T1]{fontenc}
\usepackage{amsmath, amssymb, amsthm, mathrsfs, mathtools, graphicx}
\usepackage{varioref}
\usepackage[toc,page]{appendix}
\usepackage{enumitem, xspace, ifthen, comment}
\usepackage[all]{xy}
\usepackage[dvipsnames]{xcolor}
\usepackage{tikz-cd}
\usepackage{stmaryrd}
\usepackage{braket}
\usepackage{tikz-cd}
\tikzset{
    labl/.style={anchor=south, rotate=90, inner sep=.5mm}
}
\usepackage{hyperref}
\usepackage[nameinlink, capitalize]{cleveref}
\usetikzlibrary{positioning}
\usepackage{cleveref}
\usepackage[numbers]{natbib}
\newtheorem{theorem}{Theorem}[section]
\newtheorem{proposition}[theorem]{Proposition}
\newtheorem{lemma}[theorem]{Lemma}

\newtheorem{claim}[theorem]{Claim}
\theoremstyle{definition}
\newtheorem{definition}[theorem]{Definition}
\newtheorem{example}[theorem]{Example}

\newtheorem{remark}[theorem]{Remark}

\newtheorem{notation}[theorem]{Notation}

\def\gl{\mathop{\mathrm{GL}}\nolimits}
\def\pr{\mathop{\mathrm{pr}}\nolimits}

\def\id{\mathop{\mathrm{id}}\nolimits}
\usepackage[scr=boondoxo]{mathalfa}

\title[A characterization of ball quotient stacks]{A characterization of ball quotient stacks}
\author{Chirantan Chowdhury}
\author{Matteo Costantini}
\author{Aryaman Patel}

\begin{document}

\begin{abstract}
We characterize smooth proper Deligne-Mumford stacks $\mathscr{X}$ that arise as compactifications of ball quotient stacks $[\mathbb{B}^d/\Gamma]$. Moreover, we show that every ball quotient admits a compactification whose boundary divisor $\mathscr{D}:=\mathscr{X}-[\mathbb{B}^d/\Gamma]$ is a disjoint union of quotient stacks $[A/G]$, where $A$ is an abelian variety and $G$ is a finite group. This generalizes a result of Deng-Cadorel. Our strategy combines Simpson's non-abelian Hodge correspondence for smooth proper DM-stacks, Mochizuki's generalization of the classical Simpson's correspondence to the log setting, and the uniformization results of Deng-Cadorel.
\end{abstract}

\maketitle

\section{Introduction}

The main aim of this article is to characterize smooth Deligne-Mumford stacks that are quotients of the unit ball $\mathbb{B}^d\subset\mathbb{C}^d$. This generalizes the results of Deng and Cadorel in \cite{dengcadorel}, where they characterized quasi-projective quotients of $\mathbb{B}^d$.

We make crucial use of the ideas developed by Simpson in the article \cite{simpsonstacks}, where he constructs a so-called smooth projective hypercovering $f:Z_\bullet\to\mathscr{X}$ of a smooth proper DM-stack $\mathscr{X}$ that is surjective where \'etale (see \cite[Section 5]{simpsonstacks}). 

Roughly speaking, the object $Z_\bullet$ can be viewed as a disjoint union of smooth projective varieties that cover $\mathscr{X}$. The surjective where \'etale property means that the map $f$ is surjective when restricted to the open subset of $Z_\bullet$ where it is \'etale. 

Given a simple normal crossing divisor $\mathscr{D}\subset\mathscr{X}$, we define a \emph{good covering} for the pair $(\mathscr{X},\mathscr{D})$ as a proper surjective hypercovering $f:Z_\bullet\to\mathscr{X}$ such that the reduced divisor underlying $f^*\mathscr{D}$ on $Z_\bullet$ also has simple normal crossings. Lemma \ref{Hypercover-lem} shows that a good covering always exists.

We define a bundle on $\mathscr{X}$ to be (semi)stable if its pullback to a good covering $Z_\bullet$ slope-(semi)stable with respect to some ample polarization on $Z_\bullet$. Although this notion of stability may not seem natural at first, it allows us to avoid imposing more conditions on the stack $\mathscr{X}$ (such as having a projective coarse moduli space). This is also very similar to Simpson's notion of potential (semi)stability.

The first main result of this article is the following uniformization statement, which generalizes \cite[Theorem A]{dengcadorel}.

\begin{theorem}\label{Uniformization-thm}
Let $\mathscr{X}$ be a smooth proper $d$-dimensional DM-stack of finite type, and let $\mathscr{D}\subset \mathscr{X}$ be an snc divisor on $\mathscr{X}$. Let $f:Z_\bullet\to\mathscr{X}$ be a good covering and let $\mathcal{L}$ be an ample polarizarion on $Z_\bullet$. If the log Higgs bundle $(\Omega^1_\mathscr{X}(\log \mathscr{D})\oplus\mathcal{O}_\mathscr{X},\theta)$ with Higgs field given by 
\begin{align*}
\theta:\Omega^1_{\mathscr{X}}(\log\mathscr{D})\oplus\mathcal{O}_{\mathscr{X}}&\to(\Omega^1_{\mathscr{X}}(\log\mathscr{D})\oplus\mathcal{O}_{\mathscr{X}})\otimes\Omega^1_{\mathscr{X}}(\log \mathscr{D})\\
(a,b)&\mapsto(0,1)\otimes a
\end{align*}
is polystable with respect to $\mathcal{L}$, then the Chern class inequality
\begin{align}\label{BMY1}
    (2(d+1)c_2(f^*\Omega^1_{\mathscr{X}}(\log \mathscr{D}))-dc_1(f^*\Omega^1_{\mathscr{X}}(\log \mathscr{D}))^2)\cdot c_1(\mathcal{L})^{d-2}\ge0
\end{align}
holds on $Z_\bullet$.

Moreover, if $\mathscr{D}$ is smooth and equality holds in \ref{BMY1}, then there is an isomorphism of stacks $\mathscr{X}-\mathscr{D}\cong[\mathbb{B}^d/\Gamma]$, where $\mathbb{B}^d\subset\mathbb{C}^d$ is the unit ball and $\Gamma\subset\mathrm{PU}(d,1)$ is a lattice acting on $\mathbb{B}^d$ via automorphisms.

In this case, $\mathscr{X}$ is birationally equivalent to a quotient stack $[X'/G]$, where $X'$ is the smooth toroidal compactification of a smooth ball quotient $\mathbb{B}^d/\Gamma'$, for $\Gamma'\subset\Gamma$ a normal torsion-free sublattice of finite index and $G:=\Gamma/\Gamma'$. 

If the moduli space $X$ of $\mathscr{X}$ is assumed to be projective, then $X\cong X'/G$ and $X$ is isomorphic to a projective compactification of $\mathbb{B}^d/\Gamma$ such that the boundary $D:=X-\mathbb{B}^d/\Gamma$ is isomorphic to a disjoint union of quotients of abelian varieties by finite groups.
\end{theorem}

We note that the assumption of $\mathscr{D}$ being smooth in Theorem \ref{Uniformization-thm} is indeed necessary, as remarked by Deng and Cadorel. 

\begin{remark}
If $\mathscr{D}$ is only snc but not smooth, then the universal covering stack $\widetilde{\mathscr{X}-\mathscr{D}}$ is \emph{not} isomorphic to $\mathbb{B}^d$. This is the content of \cite[Theorem 5.7(ii)]{dengcadorel}. However, it follows from Lemma \ref{Completeness-lem} that there is still an \'etale morphism $\widetilde{\mathscr{X}-\mathscr{D}}\to\mathbb{B}^d$. In particular, $\widetilde{\mathscr{X}-\mathscr{D}}$ is a smooth analytic space.
\end{remark}

When $\mathscr{D}$ is empty and $\mathscr{X}$ has trivial generic stabilizer (i.e., it is an orbifold), Theorem \ref{Uniformization-thm} has been proved in \cite{GP24}. A characterization of projective ball quotients with klt singularities has been obtained in \cite{GKPT} as a consequence of the non-abelian Hodge correspondence for klt spaces. The results in \cite{GKPT} were generalized in \cite{P23} to obtain a characterization of projective quotients of bounded symmetric domains with klt singularities. 

In order to prove Theorem \ref{Uniformization-thm}, we establish a Simpson-Mochizuki correspondence for log Higgs bundles on smooth proper DM-stacks that may be of independent interest (Theorem \ref{correspondence}). 

We also establish the following converse to Theorem \ref{Uniformization-thm}, which generalizes \cite[Theorem B]{dengcadorel}.

\begin{theorem}\label{Converse-thm}
Let $\Gamma\in\mathrm{PU}(d,1)$ be a lattice which acts on $\mathbb{B}^d$ by automorphisms. Then the quotient stack $[\mathbb{B}^d/\Gamma]$ admits a compactification $\mathscr{X}$ such that $\mathscr{X}$ is a smooth proper DM-stack and the boundary divisor $\mathscr{D}\subset\mathscr{X}$ is a disjoint union of quotients of abelian varieties by finite groups.  

Moreover, the Chern class equality
\begin{align}\label{CCeq}
2(d+1)c_2(\Omega^1_{\mathscr{X}}(\log \mathscr{D}))-dc_1(\Omega^1_{\mathscr{X}}(\log \mathscr{D}))^2=0
\end{align}
holds on $\mathscr{X}$ and the log Higgs bundle $(\Omega^1_{\mathscr{X}}(\log \mathscr{D})\oplus\mathcal{O}_{\mathscr{X}},\theta)$ defined in Theorem \ref{Uniformization-thm} is polystable with respect to any ample polarization on a good covering $Z_\bullet\to\mathscr{X}$ for the pair $(\mathscr{X},\mathscr{D})$.
\end{theorem}

\section*{Acknowledgements}
AP acknowledges support from the Deutsche Forschungsgemeinschaft (DFG, German Research Foundation) -- Project ID 286237555 (TRR 195) and Project ID 530132094. He also thanks Beno\^it Cadorel for helpful discussions. CC acknowledges support by the European Research Council (ERC) under Horizon Europe (grant agreement nº 101040935), by the Deutsche Forschungsgemeinschaft (DFG, German Research Foundation) TRR 326 \textit{Geometry and Arithmetic of Uniformized Structures}, project number 444845124 and the LOEWE professorship in Algebra, project number LOEWE/4b//519/05/01.002(0004)/87. MC was supported by the DFG Eigenestelle project  520961294, \emph{Die Geometrie der Strata von Differentialen}.

\section{Log Higgs bundles}

In this section, we introduce the objects that will be used to prove the main result, starting with log Higgs bundles. Any object on a smooth DM-stack $\mathscr{X}$ can be regarded as a collection of compatible objects on a collection of charts (which are smooth quasi-projective schemes by definition) that cover $\mathscr{X}$. For simplicity and convenience of the reader, we define all objects as defined over varieties.

A \emph{log pair} $(\mathscr{X},\mathscr{D})$ is a smooth DM-stack $\mathscr{X}$ of finite type and an snc divisor $\mathscr{D}\subset\mathscr{X}$. A log pair is called \emph{proper} if $\mathscr{X}$ is proper.

\begin{definition}
Let $(\mathscr{X},\mathscr{D})$ be a log pair. A \emph{log Higgs bundle} on $(\mathscr{X},\mathscr{D})$ is a pair $(\mathcal{E},\theta)$, where $\mathcal{E}$ is a vector bundle on $\mathscr{X}$ and $\theta:\mathcal{E}\to\mathcal{E}\otimes\Omega^1_{\mathscr{X}}(\log \mathscr{D})$ is an $\mathcal{O}_\mathscr{X}$-linear morphism such that $\theta\wedge\theta=0\in H^0(\mathscr{X},\mathcal{E}nd(\mathcal{E})\otimes\Omega^2_{\mathscr{X}}(\log \mathscr{D}))$. 
\end{definition}

When $\mathscr{D}=0$ we recover the notion of a Higgs bundle. Note that the restriction $(\mathcal{E}|_{\mathscr{X}-\mathscr{D}},\theta|_{\mathscr{X}-\mathscr{D}})$ is a Higgs bundle on $\mathscr{X}-\mathscr{D}$. 

Let $h$ be a smooth Hermitian metric on $\mathcal{E}|_{\mathscr{X}-\mathscr{D}}$. Let $d_h$ denote the Chern connection of $(\mathcal{E}|_{\mathscr{X}-\mathscr{D}},h)$ and let $\overline{\theta_h}$ denote the adjoint of $\theta$ with respect to $h$. Then the metric $h$ is called \emph{harmonic} if the curvature of the connection $D_h:=d_h+\theta_h+\overline{\theta_h}$ on $\mathcal{E}|_{\mathscr{X}-\mathscr{D}}$ is zero.

A Higgs bundle $(\mathcal{F},\phi)$ on $\mathscr{X}-\mathscr{D}$ is called \emph{harmonic} if $\mathcal{F}$ admits a harmonic metric.

\begin{definition}
    A log Higgs bundle $(\mathcal{E},\theta)$ on a log pair $(\mathscr{X},\mathscr{D})$ is said to be a \emph{system of log Hodge bundles} if $\mathcal{E}$ admits a decomposition $\mathcal{E}=\bigoplus_{p,q}\mathcal{E}^{p,q}$ such that $\theta(\mathcal{E}^{p,q})\subset\mathcal{E}^{p+1,q-1}\otimes\Omega^1_X(\log\mathscr{D})$ for all $p,q$.
\end{definition}

\begin{definition}\label{def-Hodge-metric}
    Let $(\mathcal{E}=\bigoplus_{p,q}\mathcal{E}^{p,q},\theta)$ be a system of log Hodge bundles on a log pair $(\mathscr{X},\mathscr{D})$. A hermitian metric $h$ on $\mathcal{E}$ is said to be a \emph{Hodge metric} if $h$ is harmonic and is a direct sum $h=\bigoplus_{p,q}h^{p,q}$, where $h^{p,q}$ is a metric on the bundle $\mathcal{E}^{p,q}$.
\end{definition}

We recall the definition of a Hodge group from \cite[Section 8]{Simpson_unif} and \cite[Section 3]{dengcadorel}. 

\begin{definition}
A semisimple real algebraic group $G_0$ is called a \emph{Hodge group} if the complexified Lie algebra $\mathfrak{g}=\mathrm{Lie}(G_0)\otimes_\mathbb{R}\mathbb{C}$ admits a Hodge decomposition given by 
\[
\mathfrak{g}=\bigoplus_p\mathfrak{g}^{p,-p}
\]
so that $[\mathfrak{g}^{p,-p},\mathfrak{g}^{q,-q}]\subset\mathfrak{g}^{p+q,-p-q}$ and $\overline{\mathfrak{g}^{p,-p}}=\mathfrak{g}^{-p,p}$, where $[.,.]$ denotes the Lie bracket of $\mathfrak{g}$ and $\overline{*}$ denotes complex conjugation. Moreover, the form 
\[
h_{\mathfrak{g}}(U,V):=(-1)^{p+1}\mathrm{Tr}(\mathrm{ad}(U)\mathrm{ad}(V))
\]
is a positive definite Hermitian form on $\mathfrak{g}$ for all $U,V\in\mathfrak{g}^{p,-p}$.
\end{definition}

Let $G_0$ be a Hodge group as above and let $K_0\subset G_0$ be the subgroup corresponding to the Lie algebra $\mathfrak{k}_0:=\mathfrak{g}_0\cap\mathfrak{g}^{0,0}$. We denote by $K$ the complexification of $K_0$ and by $\mathfrak{k}$ its Lie algebra. Note that $K_0$ is a compact real subgroup of $G_0$. 


\begin{definition}\label{Principal hodge bundle}
Let $(\mathscr{X},\mathscr{D})$ be a log pair. A \emph{principal system of log Hodge bundles} on $(\mathscr{X},\mathscr{D})$ for a Hodge group $G_0$ is a pair $(P,\tau)$, where $P$ is a principal $K$-bundle on $\mathscr{X}$ together with a morphism of vector bundles
\[
\tau:\mathcal{T}_{\mathscr{X}}(-\log \mathscr{D})\to P\times_K\mathfrak{g}^{-1,1}
\]
such that $[\tau(u),\tau(v)]=0$ for all $u,v\in\mathcal{T}_{\mathscr{X}}(-\log \mathscr{D})$. 
\end{definition}

The restriction $(P|_{\mathscr{X}-\mathscr{D}},\tau|_{\mathscr{X}-\mathscr{D}})$ is a principal system of Hodge bundles on $\mathscr{X}-\mathscr{D}$. A \emph{metric} for $P|_{\mathscr{X}-\mathscr{D}}$ is a principal subbundle $P_H\subset P|_{\mathscr{X}-\mathscr{D}}$ whose structure group is $K_0\subset K$ i.e., it is a reduction of structure group of $P|_{\mathscr{X}-\mathscr{D}}$ from $K$ to $K_0$. The metric $P_H$ induces a connection $D_H$ on the principal $G_0$-bundle $P_H\times_{K_0}G_0$; we refer the reader to \cite[Definition 3.2]{dengcadorel} for details.

\begin{definition}\label{Principal VHS}
Let $(\mathscr{X},\mathscr{D})$ be a log pair. Let $(P,\tau)$ be a principal system of log Hodge bundles on $(\mathscr{X},\mathscr{D})$ and let $P_H$ be a metric for $P$. Then the triple $(P|_{\mathscr{X}-\mathscr{D}},\tau|_{\mathscr{X}-\mathscr{D}},P_H)$ is a \emph{principal variation of Hodge structure} on $\mathscr{X}-\mathscr{D}$ if the connection $D_H$ on the bundle $P_H\times_{K_0}G_0$ is flat i.e., if the curvature of $D_H$ is zero.
\end{definition}

The metric reduction $P_H$ for a principal system of log Hodge bundles $(P,\tau)$ on $\mathscr{X}-\mathscr{D}$ induces a Hermitian metric $h_H$ on the bundle $P\times_K\mathfrak{g}$ given by 
\[
h_H((p,u),(p,v))=h_{\mathfrak{g}}(u,v)
\]
for any $p\in P_H$ and $u,v\in\mathfrak{g}$.

\begin{definition}
A Hodge group $G_0$ is said to be of \emph{Hermitian type} if $G_0$ has no compact factor and if the Hodge decomposition of the Lie algebra $\mathfrak{g}$ of $G$ admits a direct sum decomposition given by
\[
\mathfrak{g}=\mathfrak{g}^{-1,1}\oplus\mathfrak{g}^{0,0}\oplus\mathfrak{g}^{1,-1}
\]
In this case, $K_0\subset G_0$ is the maximal compact subgroup (unique up to conjugation) and $\mathcal{D}:=G_0/K_0$ is a Hermitian symmetric space of non-compact type (i.e. a bounded symmetric domain).
\end{definition}

\begin{definition}\label{Uniformizng VHS}
A \emph{uniformizing system of log Hodge bundles} on a log pair $(\mathscr{X},\mathscr{D})$ is a principal system of log Hodge bundles $(P,\tau)$ such that the morphism $\tau:\mathcal{T}_{\mathscr{X}}(-\log \mathscr{D})\to P\times_K\mathfrak{g}^{-1,1}$ is an isomorphism. 

Moreover, if $(P|_{\mathscr{X}-\mathscr{D}},\tau|_{\mathscr{X}-\mathscr{D}})$ admits a metric reduction $P_H$ with structure group $K_0$ such that the curvature of the connection $D_H$ on the $G_0$-bundle $P_H\times_{K_0}G_0$ is zero, then the triple $(P|_{\mathscr{X}-\mathscr{D}},\tau|_{\mathscr{X}-\mathscr{D}},P_H)$ is a \emph{uniformizing variation of Hodge structure} on $\mathscr{X}-\mathscr{D}$.
\end{definition}

\section{Proper surjective hypercoverings}

In this section, we introduce the notion of a good covering of a proper log pair and we show that such a covering always exists, following Simpson's ideas in \cite[Section 5]{simpsonstacks}.
We begin with the following analog of \cite[Theorem 5.4]{simpsonstacks}.

\begin{lemma}\label{Surj-etale-lem}
Let $\mathscr{X}$ be a smooth proper DM-stack of finite type and let $\mathscr{D}\subset \mathscr{X}$ be an snc divisor. Then there is a morphism $g:Z\to \mathscr{X}$ such that $Z$ is a smooth projective variety, $g$ is surjective where \'etale, and the reduced divisor underlying $g^*\mathscr{D}$ in $Z$ is snc.
\end{lemma}
\begin{proof}
By \cite[Theorem 5.4]{simpsonstacks}, there is a smooth projective variety $W$ and a morphism $w:W\to\mathscr{X}$ that is surjective where \'etale. In other words, there is an open subset $W'\subset W$ such that the restricted morphism $W'\to\mathscr{X}$ is an \'etale cover.

Since \'etale morphisms are smooth, the locus where the divisor $w^*\mathscr{D}$ on $W$ is not snc is contained in the complement $W-W'$. We can blow up $W$ to get a morphism $z:Z\to W$ such that $Z$ is a smooth projective variety and $z^*w^*\mathscr{D}$ is an snc divisor on $Z$. 

Let $g:=w\circ z:Z\to\mathscr{X}$ be the composition. Since $z|_{z^{-1}W'}:z^{-1}W'\to W'$ is an isomorphism by construction, we conclude that the morphism $g$ is surjective where \'etale as desired.
\end{proof}

Using Lemma \ref{Surj-etale-lem} as a starting point, we can follow Simpson's construction in \cite[Section 5]{simpsonstacks} to obtain a proper surjective hypercovering $f:Z_\bullet\to \mathscr{X}$ by smooth projective varieties such that the reduced divisor underlying $f^*\mathscr{D}$ in $Z_\bullet$ is snc. For completeness, we recall the notion of a hypercovering in the appendix. We obtain the following analog of \cite[Theorem 5.8]{simpsonstacks}.

\begin{lemma}\label{Hypercover-lem}
    Let $\mathscr{X}$ be a smooth proper DM-stack of finite type and $\mathscr{D} \subset \mathscr{X}$ be an snc divisor. Then there exists a split proper surjective hypercovering $f:Z_{\bullet} \to \mathscr{X}$ along with an \'etale subhypercovering $U_{\bullet} \to \mathscr{X}$ with the property that $(Z_k,D_k)$ is a projective log pair. Here $D_k:=f_k^{-1}(\mathscr{D})$ and $f_k: Z_k \to \mathscr{X}$ are the levels of $Z_\bullet$.
\end{lemma}
\begin{proof}
By Lemma \ref{Surj-etale-lem}, we have $f_0:(Z_0,D_0) \to (\mathscr{X},\mathscr{D})$ such that $(Z_0,D_0)$ is projective log pair and $f_0$ is surjective where \'etale i.e there exists $U_0 \subset Z_0 \to \mathscr{X}$ \'etale. \\
Let $\operatorname{PropSmDMSt}$ denote the category of smooth proper DM-stacks. Following the notation in Definition \ref{defa1}, suppose that we have constructed $(Z_{\le m}, U_{\le m}) \in \operatorname{Simp}_m(\operatorname{PropSmDMSt})$ satisfying the conditions of the theorem. We would like to construct an $m+1$ truncated object that satisfies similar conditions. We construct this in the following steps: 
\begin{enumerate}
    \item The simplicial object $\operatorname{coskn}_m (Z_{\le m})$ is a proper hypercovering and each of the terms is proper over the base (Proposition \ref{propertyofPhypercoverings}).
    \item Let $\widetilde{D}_{m+1} := \coprod_{i=0}^m (d^m_i)^{-1}(D_m)$. This is a snc divisor in $ \operatorname{coskn}_m(Z_{\le m})_{m+1}$.
    \item Let $N$ be the resolution of singularities of $(\operatorname{coskn}_m Z_{\le m})_{m+1}$ along $\widetilde{D}_{m+1}$  This comes with a map :
    \begin{equation*}
        \beta : N \to (\operatorname{coskn}_m (Z_{\le m}))_{m+1}.
    \end{equation*}
    By Proposition \ref{propertyofPhypercoverings}, let $Z_{\le m+1}$ be the $(m+1)$-split-truncated simplicial object such that $\alpha(Z_{\le m+1})= (Z_{\le m},N,\beta)$. By construction, $N$ is projective and smooth. 
    \item The map $\beta$ produces a snc divisor $D_{m+1}$ on $Z_{\le m+1}$ such that $D_{m+1}=f_{m+1}^{-1}D$.
    \item Let $U_{\le m+1}$ be the $(m+1)$-truncated simplicial object that arises from $(U_{\le m},N',\beta')$ where $\beta': \beta^{-1}(\operatorname{coskn}_m(U_{\le m})_{m+1}) \to  \operatorname{coskn}_m(U_{\le m})_{m+1}$ (see Proposition \ref{propertyofPhypercoverings}). 
    \item This completes the construction of the pair $(Z_{\le m+1},U_{\le m+1})$ along with $D_{m+1}=f_{m+1}^{-1}(D)$.
\end{enumerate}
\end{proof}

In particular, the above hypercovering also restricts to an proper surjective hypercovering on the complement $\mathscr{X}-\mathscr{D}$ via the following lemma.

\begin{lemma}\label{hypercoveropen-lem}
    The above hypercovering $f:Z_{\bullet} \to \mathscr{X}$ restricts to a surjective proper hypercovering $Z_\bullet-f^*\mathscr{D} \to \mathscr{X}-\mathscr{D}$.
\end{lemma}

\begin{proof}
    This follows from the fact that the proper surjective hypercoverings are stable under base change. In particular, one sees that $\beta|_{\operatorname{coskn}_m(U_{\le m})_{m+1}}$ is proper, and this gives us a proper surjective hypercovering. 
\end{proof}

The levels of the covering $V_\bullet:=Z_\bullet-f^*\mathscr{D}\to\mathscr{X}-\mathscr{D}$ are smooth quasi-projective varieties $V_k:=Z_k-f_k^*\mathscr{D}$. The conditions for a collection of objects $\{h_k\}$ on the levels $V_k$ to form a descent datum are the same as those in the case of a hypercovering for a smooth proper DM-stack, as formulated in \cite[Section 5]{simpsonstacks}. 

We call $f:Z_\bullet\to\mathscr{X}$ a \emph{good covering} for the pair $(\mathscr{X},\mathscr{D})$ if it satisfies the conditions in Lemma \ref{Hypercover-lem}. Note that given a log Higgs sheaf $(\mathcal{E},\mathcal{\theta})$ on $(\mathscr{X},\mathscr{D})$ and a good covering $f:Z_\bullet\to\mathscr{X}$, the pullback $(f^*\mathcal{E},f^*\theta)$ is a $f^*\Omega^1_{\mathscr{X}}(\log\mathscr{D})$-Higgs sheaf on $Z_\bullet$.

We use the following notion of semistability for log Higgs sheaves, which is inspired by Simpson's notion of potential semistability (\cite[Definition 9.1]{simpsonstacks}).

\begin{definition}\label{Potential-ss}
Let $(\mathscr{X},\mathscr{D})$ be a proper log pair. Let $f:Z_\bullet\to \mathscr{X}$ be a good covering for $(\mathscr{X},\mathscr{D})$ and let $\mathcal{L}$ be an ample line bundle on $Z_\bullet$. A log Higgs sheaf $(\mathcal{E},\theta)$ on $(\mathscr{X},\mathscr{D})$ is \emph{semistable} (resp. \emph{polystable}) with respect to $\mathcal{L}$ if $(f^*\mathcal{E},f^*\theta)$ is slope-semistable (resp. slope-polystable) with respect to $\mathcal{L}$ as a $f^*\Omega^1_{\mathscr{X}}(\log\mathscr{D})$-Higgs sheaf on $(Z_\bullet,f^*\mathscr{D})$.

This means that on each level $Z_k$ of $Z_\bullet$, $(f_k^*\mathcal{E},f_k^*\theta)$ is slope-semistable (resp. slope-polystable) with respect to $\mathcal{L}_k$ as a $f_k^*\Omega^1_{\mathscr{X}}(\log\mathscr{D})$-Higgs sheaf on $(Z_k,f_k^*\mathscr{D})$.
\end{definition}

Note that when $\mathscr{X}$ is a smooth projective variety, this notion is the same as (semi)stability with respect to an arbitrary polarization on $\mathscr{X}$, which is a reasonable condition.


\section{Variations of Hodge structure}

In this section, we prove the auxiliary statements that will be used to prove Theorem \ref{Uniformization-thm}. We begin with the following version of Simpson's correspondence for harmonic bundles, which generalizes Mochizuki's result \cite[Theorem 1.4]{mochizuki_asterisque}.

\begin{theorem}\label{correspondence}
Let $(\mathscr{X},\mathscr{D})$ be a proper log pair. Let $f:Z_\bullet\to\mathscr{X}$ be a good covering for $(\mathscr{X},\mathscr{D})$ and let $\mathcal{L}$ be an ample line bundle on $Z_\bullet$. 

A log Higgs bundle $(\mathcal{E},\theta)$ on $(\mathscr{X},\mathscr{D})$ is polystable with respect to $\mathcal{L}$ and satisfies the Chern class equality
\begin{align}\label{cherncleq}
c_1(f^*\mathcal{E})\cdot c_1(\mathcal{L})^{d-1}=ch_2(f^*\mathcal{E})\cdot c_1(\mathcal{L})^{d-2}=0 
\end{align}
if and only if there is a harmonic metric $h$ for $(\mathcal{E}|_{\mathscr{X}-\mathscr{D}},\theta|_{\mathscr{X}-\mathscr{D}})$ that is adapted to $(\mathcal{E},\theta)$. If $(\mathcal{E},\theta)$ is, moreover, stable with respect to $\mathcal{L}$, such a metric $h$ is unique up to scaling. 
\end{theorem}

\begin{proof}
We split the proof into two parts, one for each implication of the statement. 

\textbf{Step I.} Suppose that $(\mathcal{E},\theta)$ is polystable with respect to $\mathcal{L}$ and satisfies the equality \ref{cherncleq}.

The pullback $(f^*\mathcal{E},f^*\theta)$ has the structure of a log Higgs bundle on $(Z_\bullet,f^*\mathscr{D})$ via the natural morphism $\delta:f^*\Omega^1_{\mathscr{X}}(\log\mathscr{D})\to\Omega^1_{Z_\bullet}(\log f^*\mathscr{D})$ (here we abuse notation by writing $f^*\theta$ for the composition $(\id_{f^*\mathcal{E}}\otimes\delta)\circ f^*\theta$). In particular, on each level $Z_k$ of $Z_\bullet$, the pair $(f_k^*\mathcal{E},f_k^*\theta)$ has the structure of a log Higgs bundle on $(Z_k,f_k^*\mathscr{D})$. 

Note that an ample line bundle $\mathcal{L}$ on $Z_\bullet$ is the same thing as an ample line bundle $\mathcal{L}_k$ on each $Z_k$ such that the $\mathcal{L}_k$'s satisfy compatibility. The hypothesis implies that $(f_k^*\mathcal{E},f_k^*\theta)$ is slope-polystable with respect to $\mathcal{L}_k$ and that the equality $c_1(f_k^*\mathcal{E})\cdot c_1(\mathcal{L}_k)^{d-1}=ch_2(f_k^*\mathcal{E})\cdot c_1(\mathcal{L}_k)^{d-2}=0$ holds on $Z_k$ for each $k$.

Consider the decomposition 
\begin{align}\label{decompo}
(f^*\mathcal{E},f^*\theta)=\bigoplus_j(\mathcal{E}_j,\theta_j)
\end{align}
of $(f^*\mathcal{E},f^*\theta)$ as a direct sum of slope-stable log Higgs bundles $(\mathcal{E}_j,\theta_j)$ on $(Z_\bullet,f^*\mathscr{D})$. A simple Chern class computation combined with the Bogomolov-Gieseker inequality for slope-stable log Higgs bundles (\cite[Theorem 6.5]{mochizuki_asterisque}) implies that the equality
\begin{align}\label{cherneq}
c_1(\mathcal{E}_j)\cdot c_1(\mathcal{L})^{d-1}=ch_2(\mathcal{E}_j)\cdot c_1(\mathcal{L})^{d-2}=0
\end{align}
holds for each j. Restricting \ref{decompo} to each $Z_k$, we obtain a direct sum decomposition $(f_k^*\mathcal{E},f_k^*\theta)=\bigoplus_{j}(\mathcal{E}_{j,k},\theta_{j,k})$, where each $(\mathcal{E}_{j,k},\theta_{j,k})$ is a slope-stable log Higgs bundle on $(Z_k,f_k^*\mathscr{D})$. 

Set $V_\bullet:=Z_\bullet-f^*\mathscr{D}$ and $V_k:=Z_k-f_k^*\mathscr{D}$. By Lemma \ref{hypercoveropen-lem}, $f|_{V_\bullet}:V_\bullet\to\mathscr{X}-\mathscr{D}$ is a proper covering that is surjective where \'etale and the levels $V_k$ are smooth quasi-projective varieties. We apply \cite[Theorem 9.4]{mochizuki_asterisque} to each $(\mathcal{E}_{j,k},\theta_{j,k})$ to obtain harmonic metrics $h_{j,k}$ on $(\mathcal{E}_{j,k}|_{V_k},\theta_{j,k}|_{V_k})$ that are adapted to $(\mathcal{E}_{j,k},\theta_{j,k})$ in the sense of \cite[Definition 2.7]{dengcadorel}, for all $k$. Moreover, each $h_{j,k}$ is unique up to scaling. 
\begin{claim}\label{claim-descent}
For each $j$, the collection of metrics $h_{j,k}$ satisfies compatibility up to scaling, and forms a descent datum for the covering $V_\bullet\to\mathscr{X}-\mathscr{D}$ in the sense of \cite[Section 5]{simpsonstacks}. 
\end{claim}
\begin{proof}
Let $U_\bullet\subset Z_\bullet$ be the open subset restricted to which $f$ is an \'etale cover, set $E_\bullet:=Z_\bullet-U_\bullet$ and $D_\bullet:=f^*\mathscr{D}+E_\bullet\subset Z_\bullet$. Note that there is a natural injective map $\iota:\Omega^1_{Z_\bullet}(\log f^*\mathscr{D})\hookrightarrow\Omega^1_{Z_\bullet}(\log D_\bullet)$. Setting $\theta':=(\id_{f^*\mathcal{E}}\otimes\iota)\circ f^*\theta$, it is straightforward to check that $(f^*\mathcal{E},\theta')$ is a log Higgs bundle on the pair $(Z_\bullet,D_\bullet)$.

It follows from the definition of $\theta'$ that every $\theta'$-invariant subsheaf of $f^*\mathcal{E}$ is also $f^*\theta$-invariant and thus $(f^*\mathcal{E},\theta')$ is also slope-polystable with respect to $\mathcal{L}$. In particular, $(\mathcal{E}_j,\theta'_j:=\theta'|_{\mathcal{E}_j})$ is a slope-stable log Higgs bundle on $(Z_\bullet,D_\bullet)$ for each $j$, and we have the decomposition 
\[
(f^*\mathcal{E},\theta')=\bigoplus_j(\mathcal{E}_j,\theta'_j)
\]
of $(f^*\mathcal{E},\theta')$ into stable direct summands. 

Let $W_\bullet:=Z_\bullet-D_\bullet$, $W_k:=Z_k-D_k$, and note that $W_\bullet=U_\bullet\cap V_\bullet\subset V_\bullet$ is the locus where $f|_{V_\bullet}$ is an \'etale cover. From the Chern class equality \ref{cherneq} and the slope stability of each $(\mathcal{E}_j,\theta'_j)$ with respect to $\mathcal{L}$, it follows again from \cite[Theorem 9.4]{mochizuki_asterisque} that for each $k$, there are harmonic metrics $h'_{j,k}$ on the Higgs bundles $(\mathcal{E}_{j,k}|_{W_k},\theta'_{j,k}|_{W_k})$ that are adapted to $(\mathcal{E}_{j,k},\theta'_{j,k})$ and are unique up to scaling.

Uniqueness up to scaling implies that we have $h'_{j,k}=h_{j,k}|_{W_k}$ for each $k$. Since $W_\bullet\to\mathscr{X}-\mathscr{D}$ is an \'etale cover, it follows that $W_k\to\mathscr{X}-\mathscr{D}$ is an \'etale cover for each $k$. Let $\pr_1,\pr_2$ denote the two projections $V_0\times_{\mathscr{X}-\mathscr{D}}V_0\to V_0$ and let $\pr'_1,\pr'_2$ denote the two projections $W_0\times_{\mathscr{X}-\mathscr{D}}W_0\to W_0$, which are \'etale. 

For each $j$, we have that $(\pr'_1)^*h'_{j,0}$ and $(\pr'_2)^*h'_{j,0}$ are both harmonic metrics on $(\mathcal{E}_{j,1}|_{W_1},\theta'_{j,1}|_{W_1})$ that agree with $h'_{0,1}$ up to scaling. This means that the metrics $(\pr_1)^*h_{j,0}$ and $(\pr_2)^*h_{j,0}$ on $(\mathcal{E}_{j,1}|_{V_1},\theta_{j,1}|_{W_1})$ agree with metric $h_{j,1}$ up to scaling because they agree on the dense open subset $W_1\subset V_1$. It follows that there is a collections of scalars $\lambda_{j,k}$ for each $j$ such that $\{\lambda_{j,k}h_{j,k}\}$ is a collection of harmonic metrics on $(\mathcal{E}_{j,k}|_{V_k},\theta_{j,k}|_{V_k})$ that are adapted to $(\mathcal{E}_{j,k},\theta_{j,k})$ and satisfy compatibility (in other words, for each $j$, the Higgs bundle $(\mathcal{E}_j|_{V_\bullet},\theta_j|_{V_\bullet})$ is \emph{strongly polarizable} in the sense of \cite[Section 7]{simpsonstacks}). Thus, for each $j$, the collection $\{\lambda_{j,k}h_{j,k}\}$ forms a descent datum for $V_\bullet\to\mathscr{X}-\mathscr{D}$ as claimed.   
\end{proof}
Hence, for each $j$, there is a harmonic metric $h_{j\bullet}$ on the Higgs bundle $(\mathcal{E}_j|_{V_\bullet},\theta_j|_{V_\bullet})$ that is adapted to $(\mathcal{E}_j,\theta_j)$. Moreover, $h_{j\bullet}$ is unique up to scaling because each $h_{j,k}$ is unique up to scaling. We have that $h_\bullet:=\bigoplus_jh_{j\bullet}$ is a harmonic metric on $(f^*\mathcal{E}|_{V_\bullet},f^*\theta|_{V_\bullet})$ that is adapted to $(f^*\mathcal{E},f^*\theta)$. The surjective where \'etale property of $f$ implies that $h_\bullet$ descends to a harmonic metric $h$ on $(\mathcal{E}|_{\mathscr{X}-\mathscr{D}},\theta|_{\mathscr{X}-\mathscr{D}})$ that is adapted to $(\mathcal{E},\theta)$. This concludes the first step of the proof. 

\textbf{Step II.} Suppose that there is a harmonic metric $h$ on $(\mathcal{E}|_{\mathscr{X}-\mathscr{D}},\theta|_{\mathscr{X}-\mathscr{D}})$ that is adapted to $(\mathcal{E},\theta)$. This pulls back to a harmonic metric $h_\bullet$ on the bundle $(f^*\mathcal{E}|_{V_\bullet},f^*\theta|_{V_\bullet})$ that is adapted to $(f^*\mathcal{E},f^*\theta)$. By definition, this is equivalent to a collection of harmonic metrics $\{h_k\}$ on $(f_k^*\mathcal{E}|_{V_k},f_k^*\theta|_{V_k})$ that are adapted to $(f_k^*\mathcal{E},f_k^*\theta)$ and satisfy compatibility.

By \cite[Propositions 5.1-5.3]{mochizuki_asterisque}, it follows that each $(f_k^*\mathcal{E},f_k^*\theta)$ is slope-polystable with respect to $\mathcal{L}_k$ as a log Higgs bundle on $(Z_k,f_k^*\mathscr{D})$ and satisfies the Chern class equality $c_1(f_k^*\mathcal{E})\cdot c_1(\mathcal{L}_k)^{d-1}=ch_2(f_k^*\mathcal{E})\cdot c_1(\mathcal{L}_k)^{d-2}=0$. By Definition \ref{Potential-ss}, this means that $(\mathcal{E},\theta)$ is polystable with respect to $\mathcal{L}$ and satisfies 
\[
c_1(f^*\mathcal{E})\cdot c_1(\mathcal{L})^{d-1}=ch_2(f^*\mathcal{E})\cdot c_1(\mathcal{L})^{d-2}=0, 
\]
which is precisely \ref{cherncleq}. This concludes the proof of the converse direction and hence of the Theorem.
\end{proof}

Next, we prove the following generalization of \cite[Proposition 2.12]{dengcadorel}, which is a refinement of the previous Theorem \ref{correspondence}. It says that a polystable system of log Hodge bundles on $(\mathscr{X},\mathscr{D})$ that satisfies the Chern class equality \ref{cherncleq} decomposes as a direct sum of stable systems of log Hodge bundles, each of which admits a Hodge metric (Definition \ref{def-Hodge-metric}) over $\mathscr{X}-\mathscr{D}$.

\begin{proposition}\label{Hodge-metric}
Let $(\mathscr{X},\mathscr{D})$ be a proper $d$-dimensional log pair. Let $f:Z_\bullet\to\mathscr{X}$ be a good covering for $(\mathscr{X},\mathscr{D})$ and let $\mathcal{L}$ be an ample line bundle on $Z_\bullet$. Suppose that $(\mathcal{E},\theta)$ is a system of log Hodge bundles on $(\mathscr{X},\mathscr{D})$ such that $(\mathcal{E},\theta)$ is polystable with respect to $\mathcal{L}$ and satisfies $c_1(f^*\mathcal{E})\cdot c_1(\mathcal{L})^{d-1}=ch_2(f^*\mathcal{E})\cdot c_1(\mathcal{L})^{d-2}=0$. 

Then there is a decomposition $(\mathcal{E},\theta)=\bigoplus_{i\in I}(\mathcal{E}_i,\theta_i)$, where each $(\mathcal{E}_i,\theta_i)$ is a system of log Hodge bundles on $(\mathscr{X},\mathscr{D})$ that is stable with respect to $\mathcal{L}$ and such that there is a Hodge metric $h_i$ (unique up to scaling) for $(\mathcal{E}_i|_{\mathscr{X}-\mathscr{D}},\theta_i|_{\mathscr{X}-\mathscr{D}})$ that is adapted to $(\mathcal{E}_i,\theta_i)$.
\end{proposition}

\begin{proof}
The hypothesis implies that each pair $(f_k^*\mathcal{E},f_k^*\theta)$ is a system of log Hodge bundles on $(Z_k,f_k^*\mathscr{D})$ that is slope-polystable with respect to $\mathcal{L}_k$ and satisfies $c_1(f_k^*\mathcal{E})\cdot c_1(\mathcal{L}_k)^{d-1}=ch_2(f_k^*\mathcal{E})\cdot c_1(\mathcal{L}_k)^{d-2}=0$.

Let $V_\bullet:=Z_\bullet-f^*\mathscr{D}$ and $V_k:=Z_k-f_k^*\mathscr{D}$ as before. We first consider the case where $(f^*\mathcal{E},f^*\theta)$ is slope-stable with respect to $\mathcal{L}$. It follows from \cite[Theorem 9.4]{mochizuki_asterisque} applied to each $(f_k^*\mathcal{E},f_k^*\theta)$ that there are harmonic metrics $h_k$ on $(f_k^*\mathcal{E}|_{V_k},f_k^*\theta|_{V_k})$ that are adapted to $(f_k^*\mathcal{E},f_k^*\theta)$ and are unique up to scaling. Claim \ref{claim-descent} says that the collection of metrics $\{h_k\}$ satisfies compatibility and thus gives a metric $h_\bullet$ on $(f^*\mathcal{E}|_{V_\bullet},f^*\theta|_{V_\bullet})$ that is adapted to $(f^*\mathcal{E},f^*\theta)$ and is unique up to scaling. This descends to a harmonic metric $h$ on $(\mathcal{E}|_{\mathscr{X}-\mathscr{D}},\theta|_{\mathscr{X}-\mathscr{D}})$ that is adapted to $(\mathcal{E},\theta)$ and is unique up to scaling. 

The computation in the proof of \cite[Proposition 2.12]{dengcadorel} implies that each $h_k$ is a Hodge metric and hence the same holds for the metrics $h_\bullet$ and $h$. This proves the proposition in the case that $(\mathcal{E},\theta)$ is stable with respect to $\mathcal{L}$.  

We now consider the general case. It follows from \cite[Corollary 3.11]{mochizuki_asterisque} that for each $k$, there is a unique direct sum decomposition 
\begin{align}\label{can-decomp}
(f_k^*\mathcal{E},f_k^*\theta)=\bigoplus_{i\in I}(\mathcal{E}_{i,k},\theta_{i,k}),
\end{align}
such that each $(\mathcal{E}_{i,k},\theta_{i,k})$ is a log Higgs bundle on $(Z_k,f_k^*\mathscr{D})$ that is slope-stable with respect to $\mathcal{L}_k$. From the proof of \cite[Proposition 2.12]{dengcadorel} it follows that each $(\mathcal{E}_{i,k},\theta_{i,k})$ in \ref{can-decomp} is in fact a system of log Hodge bundles. 

By the uniqueness of the decompositions \ref{can-decomp}, it follows that for each $i$, the collection of summands $\{(\mathcal{E}_{k,i},\theta_{k,i})\}$ satisfies compatibility and thus gives a system of log Hodge bundles $(\mathcal{E}_{i\bullet},\theta_{i\bullet})$ on $(Z_\bullet,f^*\mathscr{D})$ that is slope-stable with respect to $\mathcal{L}$. Thus on $Z_\bullet$, we obtain a direct sum decomposition 
\begin{align}\label{decomp}
(f^*\mathcal{E},f^*\theta)=\bigoplus_{i\in I}(\mathcal{E}_{i\bullet},\theta_{i\bullet}).
\end{align}
The surjective where \'etale property of $f$ implies that \ref{decomp} descends to a decomposition 
\[
(\mathcal{E},\theta)=\bigoplus_{i\in I}(\mathcal{E}_i,\theta_i),
\]
on $\mathscr{X}$, where for each $i$, $(\mathcal{E}_i,\theta_i)$ is a system of log Hodge bundles on $(\mathscr{X},\mathscr{D})$ that is stable with respect to $\mathcal{L}$.

By \cite[Theorem 9.4]{mochizuki_asterisque} again, there are harmonic metrics $h_{i,k}$ for $(\mathcal{E}_{i,k}|_{V_k},\theta_{i,k}|_{V_k})$, unique up to scaling, which are adapted to $(\mathcal{E}_{i,k},\theta_{i,k})$. From the first part of the proof, we have that each $h_{i,k}$ is in fact a Hodge metric on $(\mathcal{E}_{i,k}|_{V_k},\theta_{i,k}|_{V_k})$. We know from Claim \ref{claim-descent} that for each $i$, the collection of metrics $\{h_{i,k}\}$ satisfies compatibility to give a Hodge metric $h_{i\bullet}$ on $(\mathcal{E}_{i\bullet}|_{V_\bullet},\theta_{i\bullet}|_{V_\bullet})$ that is unique up to scaling and adapted to $(\mathcal{E}_{i\bullet},\theta_{i\bullet})$. This descends to a Hodge metric $h_i$ on $(\mathcal{E}_i|_{\mathscr{X}-\mathscr{D}},\theta_i|_{\mathscr{X}-\mathscr{D}})$ that is unique up to scaling and adapted to $(\mathcal{E},\mathcal{\theta})$, as required.
This concludes the proof.
\end{proof}

The following Lemma, which is an analog of \cite[Lemma 3.10]{dengcadorel}, says that a Hodge representation together with a principal system of log Hodge bundles corresponds to a system of log Hodge bundles. 

\begin{lemma}\label{VHS-lemma}
Let $(\mathscr{X},\mathscr{D})$ be a log pair. If $\rho:G\to\gl(V)$ is a Hodge representation of the Hodge group $G_0$ and $(P,\tau)$ is a principal system of log Hodge bundles on $(\mathscr{X},\mathscr{D})$, then $(\mathcal{E}:=P\times_KV,\theta:=\mathrm{d}\rho(\tau))$ is a system of log Hodge bundles. A polarization $h_V$ for $V$ together with a metric $P_H$ for $P|_{\mathscr{X}-\mathscr{D}}$ give a metric $h_{\mathcal{E}}$ for the system of Hodge bundles $(\mathcal{E}|_{\mathscr{X}-\mathscr{D}},\theta|_{\mathscr{X}-\mathscr{D}})$ on $\mathscr{X}-\mathscr{D}$. 

If $(P|_{\mathscr{X}-\mathscr{D}},\tau|_{\mathscr{X}-\mathscr{D}},P_H)$ is a principal variation of Hodge structures on $\mathscr{X}-\mathscr{D}$, then $(\mathcal{E}|_{\mathscr{X}-\mathscr{D}},\theta|_{\mathscr{X}-\mathscr{D}},h_{\mathcal{E}})$ is a complex variation of Hodge structure.
\end{lemma}
\begin{proof}
The strategy is to first apply \cite[Lemma 3.10]{dengcadorel} to local \'etale charts of $\mathscr{X}$ and then glue.

Let $\{g_k:U_k\to \mathscr{X}\}$ be a collection of compatible \'etale charts for $(\mathscr{X},\mathscr{D})$. Then each $(U_k,g_k^*\mathscr{D})$ is a quasi-projective log pair and $(g_k^*P,g_k^*\tau)$ is a principal system of log Hodge bundles on $(U_k,g_k^*\mathscr{D})$. Applying \cite[Lemma 3.10]{dengcadorel} to each $(g_k^*P,g_k^*\tau)$, it follows that each $(\mathcal{E}_k:=g_k^*P\times_KV,\theta_k:=\mathrm{d}\rho(g_k^*\tau))$ is a system of log Hodge bundles on $(U_k,g_k^*\mathscr{D})$. Note that the collection $\{(\mathcal{E}_k,\theta_k)\}$ satisfies compatibility because $\{(g_k^*P,g_k^*\tau)\}$ does. Hence, they glue to give a system of log Hodge bundles $(\mathcal{E}=P\times_KV,\theta=\mathrm{d}\rho(\tau))$ on $(\mathscr{X},\mathscr{D})$.

Let $h_V$ be a polarization for $V$ and let $P_H$ be a metric for $P|_{\mathscr{X}-\mathscr{D}}$. This is the same as a collection of metrics $P_{H_k}$ for $g_k^*P|_{U_k-g_k^*\mathscr{D}}$) that satisfy compatibility. By \cite[Lemma 3.10]{dengcadorel} again, we obtain a collection of metrics $h_{\mathcal{E}_k}$ for $\mathcal{E}_k|_{U_k-g_k^*\mathscr{D}}$. Note that the collection $\{h_{\mathcal{E}_k}\}$ satisfies compatibility because $\{P_{H_k}\}$ does. Hence, we can glue to obtain a metric $h_{\mathcal{E}}$ for the system of Hodge bundles $(\mathcal{E}|_{\mathscr{X}-\mathscr{D}},\theta|_{\mathscr{X}-\mathscr{D}})$ on $\mathscr{X}-\mathscr{D}$.

Let $(P|_{\mathscr{X}-\mathscr{D}},\tau|_{\mathscr{X}-\mathscr{D}},P_H)$ be a principal variation of Hodge structure on $\mathscr{X}-\mathscr{D}$. This is the same as a collection of principal variations of Hodge structure $(g_k^*P|_{U_k-g_k^*\mathscr{D}},g_k^*\tau|_{U_k-g_k^*\mathscr{D}},P_{H,k})$ on $U_k-g_k^*\mathscr{D}$ that satisfies compatibility. Then it follows from \cite[Lemma 3.10]{dengcadorel} that each triple $(\mathcal{E}_k|_{U_k-g_k^*\mathscr{D}},\theta_k|_{U_k-g_k^*\mathscr{D}},h_{\mathcal{E}_k})$ is a complex variation of Hodge structure on $U_k-g_k^*\mathscr{D}$ associated to the representation $\rho$ and the collection of these satisfies compatibility. Thus, we conclude that $(\mathcal{E}|_{\mathscr{X}-\mathscr{D}},\theta|_{\mathscr{X}-\mathscr{D}},h_{\mathcal{E}})$ is a complex variation of Hodge structure on $\mathscr{X}-\mathscr{D}$.
\end{proof}

We also have the following converse, which says that to system of log Hodge bundles we can associate a principal one. This is an analog of \cite[Proposition 3.11]{dengcadorel} adapted to our setting.

\begin{proposition}\label{Sys-of-Hodge-bundles-lem}
Let $(\mathscr{X},\mathscr{D})$ be a log pair and let $(\mathcal{E}=\bigoplus_{p+q=w}\mathcal{E}^{p,q},\theta)$ be a system of log Hodge bundles on $(\mathscr{X},\mathscr{D})$. Then there is a principal system of log Hodge bundles $(P,\tau)$ with structure group $K:=\mathrm{P}(\prod_p\gl(r_p,\mathbb{C}))$ associated to $(\mathcal{E},\theta)$, where $r_p:=\mathrm{rank}(\mathcal{E}^{p,q})$. Moreover, any Hermitian metric $h=\bigoplus_{p+q=w}h_p$ for $\mathcal{E}|_{\mathscr{X}-\mathscr{D}}$ gives rise to a metric reduction $P_H$ for $P|_{\mathscr{X}-\mathscr{D}}$ with the structure group $K_0:=\mathrm{P}(\prod_pU(r_p))\subset K$. 
\end{proposition}

\begin{proof}
The decomposition $\mathcal{E}=\bigoplus_{p+q=w}\mathcal{E}^{p,q}$ implies that the frame bundle of $\mathcal{E}$ admits a reduction in structure group $P'$ from $\gl(r,\mathbb{C})$ to $\prod_p\gl(r_p,\mathbb{C})$, where $r$ and $r_p$ denote the ranks of $\mathcal{E}$ and $\mathcal{E}^{p,q}$, respectively. We let $K=\mathrm{P}(\prod_p\gl(r_p,\mathbb{C}))$ be the image of $\prod_p\gl(r_p,\mathbb{C})$ via the natural morphism $\gl(r,\mathbb{C})\to\mathrm{P}\gl(r,\mathbb{C})$.  

Let $P$ be the principal $K$-bundle obtained by extending the structure group of $P'$ from $\prod_p\gl(r_p,\mathbb{C})$ to $K$, and set $\tau:=\theta$. Then $(P,\tau)$ is the desired principal system of log Hodge bundles associated to $(\mathcal{E},\theta)$.

Following the proof of \cite[Proposition 3.11]{dengcadorel}, let $Q$ be the sesquilinear form 
\[
Q(u,v)=(\sqrt{-1})^{p-q}h(u,v), 
\]
for $u,v\in\mathcal{E}^{p,q}$. Let $P'_H$ be the reduction in structure group of the frame bundle of $\mathcal{E}|_{\mathscr{X}-\mathscr{D}}$ consisting of unitary frames with respect to $Q$. Then the structure group of $P'_H$ is $\prod_{p+q=w}\mathrm{U}(r_p)$. 

Let $P_H$ be the extension of structure group of $P'_H$ to $K_0=\mathrm{P}(\prod_pU(r_p))$ obtained via the morphism $\gl(r,\mathbb{C})\to\mathrm{P}\gl(r,\mathbb{C})$. Then $P_H\subset P|_{\mathscr{X}-\mathscr{D}}$ is the desired metric reduction. 
\end{proof}

We have the following version of \cite[Theorem 4.1]{dengcadorel} in our setting, which says that a principal variation of Hodge structure can be constructed from a polystable system of log Hodge bundles with vanishing Chern classes.

\begin{proposition}\label{Principal-VHS-prop}
Let $(\mathscr{X},\mathscr{D})$ be a proper log pair, let $f:Z_\bullet\to\mathscr{X}$ be a good covering for $(\mathscr{X},\mathscr{D})$, and let $\mathcal{L}$ be an ample line bundle on $Z_\bullet$. Let $(P,\tau)$ be a principal system of log Hodge bundles on $(\mathscr{X},\mathscr{D})$ and let $\rho:G\to\gl(V)$ be a Hodge representation for some polarized Hodge structure $(V,h_V)$ so that $\rho|_{K_0}:K_0\to \gl(V)$ is faithful and $\mathrm{d}\rho:\mathfrak{g}\to\mathfrak{gl}(V)$ is injective. 

If the system of log Hodge bundles $(\mathcal{E}:=P\times_KV,\theta=\mathrm{d}\rho(\tau))$ is polystable with respect to $\mathcal{L}$ and satisfies $ch_2(\mathcal{E})\cdot c_1(\mathcal{L})^{d-2}=0$, then there is a metric reduction $P_H$ for $P|_{\mathscr{X}-\mathscr{D}}$ so that the triple $(P|_{\mathscr{X}-\mathscr{D}},\tau|_{\mathscr{X}-\mathscr{D}},P_H)$ is a principal variation of Hodge structures on $\mathscr{X}-\mathscr{D}$. Moreover, such a $P_H$ together with the polarization $h_V$ for $V$ gives rise to a Hodge metric $h$ for $(\mathcal{E}|_{\mathscr{X}-\mathscr{D}},\theta|_{\mathscr{X}-\mathscr{D}})$ which is adapted to $(\mathcal{E},\theta)$.
\end{proposition}
\begin{proof}
Following the proof of \cite[Theorem 4.1]{dengcadorel}, we first show that $(\mathcal{E}|_{\mathscr{X}-\mathscr{D}},\theta|_{\mathscr{X}-\mathscr{D}})$ admits a Hodge metric that is adapted to $(\mathcal{E},\mathcal{\theta})$. This amounts to showing that $(f^*\mathcal{E}|_{V_\bullet},f^*\theta|_{V_\bullet})$ admits a Hodge metric that is adapted to $(f^*\mathcal{E},f^*\theta)$ and descends to a Hodge metric on $(\mathcal{E}|_{\mathscr{X}-\mathscr{D}},\theta|_{\mathscr{X}-\mathscr{D}})$, where $V_\bullet:=Z_\bullet-f^*\mathscr{D}$ as before.

We recall from the proof of \cite[Theorem 4.1]{dengcadorel} that $K$ being a complex semi-simple Lie group implies that the representation $\rho':K\to\gl(\det V)$ induced by $\rho$ has image contained in $\mathrm{SL}(\det V)=\{1\}$. Hence $\rho'$ is trivial and the line bundle $\det(f^*\mathcal{E})=f^*P\times_K\det V$ is the trivial line bundle on $Z_\bullet$. Therefore, we have $c_1(f^*\mathcal{E})=0$. Together with the assumption that $(f^*\mathcal{E},f^*\theta)$ is slope-polystable with respect to $\mathcal{L}$ and satisfies $ch_2(\mathcal{E})\cdot c_1(\mathcal{L})^{d-2}=0$, Proposition \ref{Hodge-metric} says that $(\mathcal{E}|_{\mathscr{X}-\mathscr{D}},\theta|_{\mathscr{X}-\mathscr{D}})$ admits a Hodge metric that is adapted to $(\mathcal{E},\mathcal{\theta})$.

Recall that the levels of $Z_\bullet$ are smooth projective varieties and $(Z_k,f_k^*\mathscr{D})$ are projective log pairs. Consider the principal systems of log Hodge bundles $(f_k^*P,f_k^*\tau)$ on $(Z_k,f_k^*\mathscr{D})$ and let $T^{a,b}V:=\mathrm{Hom}(V^{\otimes a},V^{\otimes b})$. Then the representation $\rho$ induces a representation $\rho^{a,b}:G\to\gl(T^{a,b}V)$ for any $a,b\in\mathbb{N}$. 

It follows from the Tannakian arguments in the proof of \cite[Theorem 4.1]{dengcadorel} that there is a one-dimensional subspace $V_1\in T^{a,b}V$ such that 
\[
K=\{g\in\gl(V)\;|\; \rho^{a,b}(g)(V_1)=V_1\}
\]
for some $a,b\in\mathbb{N}$. The complement $V_2$ of $V_1$ in $T^{a,b}V$ is also $K$-invariant. Lemma \ref{VHS-lemma} implies that for each $k$, the Hodge representation $\rho^{a,b}$ together with $(f_k^*P,f_k^*\tau)$ give rise to a system of log Hodge bundles
\[
T^{a,b}(f_k^*\mathcal{E},f_k^*\theta)=(f_k^*P\times_KT^{a,b}V,\theta_k^{a,b}:=\mathrm{d}\rho^{a,b}(f_k^*\tau)).
\]
Setting $(\mathcal{E}_{i,k},\theta_{i,k}):=(f_k^*P\times_KV_i,\mathrm{d}\rho^{a,b}(f_k^*\tau))$ for $i\in\{1,2\}$, we have 
\[
T^{a,b}(f_k^*\mathcal{E},f_k^*\theta)\cong(\mathcal{E}_{1,k},\theta_{1,k})\oplus(\mathcal{E}_{2,k},\theta_{2,k})
\]
for each $k$. Theorem \ref{correspondence} implies that $T^{a,b}(f_k^*\mathcal{E},f_k^*\theta)$ is slope-polystable with respect to $\mathcal{L}_k$ and satisfies $c_1(T^{a,b}f_k^*\mathcal{E})\cdot c_1(\mathcal{L}_k)^{d-1}=0$, from which it follows that $(\mathcal{E}_{i,k},\theta_{i,k})$ is polystable with respect to $\mathcal{L}_k$, for $i\in\{1,2\}$.

Proposition \ref{Hodge-metric} then implies that each $(\mathcal{E}_{i,k}|_{V_k},\theta_{i,k}|_{V_k})$ admits a Hodge metric $h_{i,k}$ that is adapted to $(\mathcal{E}_{i,k},\theta_{i,k})$ and $f_k^*h$ coincides with $h_{1,k}\oplus h_{2,k}$ up to an obvious ambiguity. We view the principal $K$-bundle $f_k^*P|_{V_k}$ as a reduction of structure group of the frame bundle of $f_k^*\mathcal{E}|_{V_k}$ from $\gl(V)$ to $K$. The metric $f_k^*h$ for $(f_k^*\mathcal{E}|_{V_k},\theta|_{V_k})$ gives rise to a metric reduction $P_{\mathrm{U}(f_k^*\mathcal{E},f_k^*h)}$ of the frame bundle, and the structure group of $P_{\mathrm{U}(f_k^*\mathcal{E},f_k^*h)}$ is $\mathrm{U}(V,h_V)$. Note that the collection of bundles $\{P_{\mathrm{U}(f_k^*\mathcal{E},f_k^*h)}\}$ satisfies compatibility because the collection of metrics $\{f_k^*h\}$ does.

For each $k$, define the principal bundle $P_{H,k}:=P_{\mathrm{U}(f_k^*\mathcal{E},f_k^*h)}\cap f_k^*P|_{V_k}$ which has structure group $\mathrm{U}(V,h_V)\cap K=K_0$. Thus, each $P_{H,k}\subset P|_{V_k}$ is a metric reduction with structure group $K_0$. The collection $\{P_{H,k}\}$ satisfies compatibility because $\{P_{\mathrm{U}(f_k^*\mathcal{E},f_k^*h)}\}$ and $\{f_k^*P|_{V_k}\}$ do, and we obtain a metric reduction $P_H\subset P|_{\mathscr{X}-\mathscr{D}}$ as claimed. Moreover, $P_H$ corresponds to the metric $h$ on $(\mathcal{E}|_{\mathscr{X}-\mathscr{D}},\theta|_{\mathscr{X}-\mathscr{D}})$.

The arguments in the last part of the proof of \cite[Theorem 4.1]{dengcadorel} show that each triple $(f_k^*P|_{V_k},f_k^*\tau|_{V_k},P_{H_k})$ is a principal variation of Hodge structure on $V_k$. The collection $(f_k^*P|_{V_k},f_k^*\tau|_{V_k},P_{H_k})$ satisfies compatibility, since $\{P_{H,k}\}$ does, and we obtain a principal variation of Hodge structure $(f^*P|_{V_\bullet},f^*\tau|_{V_\bullet},P_{H\bullet})$ on $V_\bullet$. The surjective where \'etale property of $f|_{V_\bullet}:V_\bullet\to\mathscr{X}-\mathscr{D}$ implies that this descends to a principal variation of Hodge structure $(P|_{\mathscr{X}-\mathscr{D}},\tau|_{\mathscr{X}-\mathscr{D}},P_H)$ on $\mathscr{X}-\mathscr{D}$. This finishes the proof.
\end{proof}

\section{Uniformization}

In this section, we prove Theorems \ref{Uniformization-thm} and \ref{Converse-thm}. We begin by observing that a uniformizing system of log Hodge bundles on $\mathscr{X}$ for a Hodge group $G_0$ of Hermitian type that restricts to a uniformizing variation of Hodge structure on $\mathscr{X}-\mathscr{D}$ induces an \'etale morphism $\widetilde{\mathscr{X}-\mathscr{D}}\to G_0/K_0$ of analytic stacks, where $\widetilde{\mathscr{X}-\mathscr{D}}$ denotes the universal covering stack of $\mathscr{X}-\mathscr{D}$ and $K_0$ is the maximal compact subgroup of $G_0$. The quotient $G_0/K_0$ is a bounded symmetric domain $\mathcal{D}$.  

\begin{lemma}\label{Completeness-lem}
Let $(\mathscr{X},\mathscr{D})$ be a proper log pair and let $(P,\tau)$ be a uniformizing system of log Hodge bundles on $(\mathscr{X},\mathscr{D})$ for a Hodge group $G_0$ of Hermitian type, as in Definition \ref{Uniformizng VHS}. Suppose that there is a metric reduction $P_H\subset P|_{\mathscr{X}-\mathscr{D}}$ such that $(P|_{\mathscr{X}-\mathscr{D}},\tau|_{\mathscr{X}-\mathscr{D}},P_H)$ is a uniformizing variation of Hodge structure on $\mathscr{X}-\mathscr{D}$.

If the metric induced by $P_H$ on $\mathcal{T}_{\mathscr{X}-\mathscr{D}}$ is complete, then there is an isomorphism $\widetilde{\mathscr{X}-\mathscr{D}}\xrightarrow{\sim}G_0/K_0=:\mathcal{D}$.
\end{lemma}
\begin{proof}
Let $(\widetilde{P},\widetilde{\tau},\widetilde{P}_H)$ denote the pullback of the uniformizing variation of Hodge structure $(P|_{\mathscr{X}-\mathscr{D}},\tau|_{\mathscr{X}-\mathscr{D}},P_H)$ to the universal covering stack $\widetilde{\mathscr{X}-\mathscr{D}}$.

The flat connection $D_H$ on the bundle $P_H\times_{K_0}G_0$ (see Definition \ref{Uniformizng VHS}) pulls back to a flat connection on the bundle $\widetilde{P}_H\times_{K_0}G_0$ and induces a trivialization $\widetilde{P}_H\times_{K_0}G_0\cong(\widetilde{\mathscr{X}-\mathscr{D}})\times G_0$. Consider the composition
\[
\widetilde{P}_H\hookrightarrow\widetilde{P}_H\times_{K_0}G_0\xrightarrow{\sim}(\widetilde{\mathscr{X}-\mathscr{D}})\times G_0\to G_0
\]
which is a $K_0$-equivariant morphism. This induces a morphism 
\[
\phi:\widetilde{\mathscr{X}-\mathscr{D}}\to G_0/K_0=:\mathcal{D}.
\] 
By construction, we have $\widetilde{P}_H\cong\phi^*G_0$, where $G_0$ is viewed as a principal $K_0$-bundle on $\mathcal{D}$.

The metric $h_H$ on $P|_{\mathscr{X}-\mathscr{D}}\times_K\mathfrak{g}^{-1,1}$ pulls back to a metric $\widetilde{h}_H$ on $\widetilde{P}\times_K\mathfrak{g}^{-1,1}$. Note that this is the same as the metric induced by the reduction $\widetilde{P}_H\subset\widetilde{P}$. Since $\widetilde{\tau}:\mathcal{T}_{\widetilde{\mathscr{X}-\mathscr{D}}}\to\widetilde{P}\times_K\mathfrak{g}^{-1,1}$ is a isomorphism by Definition \ref{Uniformizng VHS}, this may be viewed as a metric on $\mathcal{T}_{\widetilde{\mathscr{X}-\mathscr{D}}}$.

On $\mathcal{D}$, we have an isomorphism $\mathcal{T}_{\mathcal{D}}\cong G_0\times_{K_0}\mathfrak{g}^{-1,1}$, and since $\widetilde{P}_H\cong\phi^*G_0$, it follows that $\mathcal{T}_{\widetilde{\mathscr{X}-\mathscr{D}}}\cong\phi^*\mathcal{T}_{\mathcal{D}}$. This implies that $\phi$ is \'etale and since $\mathcal{D}$ is an analytic space, so is $\widetilde{\mathscr{X}-\mathscr{D}}$.

Moreover, we have $\widetilde{h}_H=\phi^*h_\mathcal{D}$, where $h_{\mathcal{D}}$ is the canonical Hermitian metric on $\mathcal{T}_{\mathcal{D}}$. Thus $\phi$ is in fact a local isometry and $\widetilde{h}_H$ in invariant under the $\pi_1(\mathscr{X}-\mathscr{D})$-action on $\widetilde{\mathscr{X}-\mathscr{D}}$.

Therefore, if $h_H$ is a complete metric, then so is $\widetilde{h}_H$ and $\phi:\widetilde{\mathscr{X}-\mathscr{D}}\to\mathcal{D}$ is a covering map. Since $\widetilde{\mathscr{X}-\mathscr{D}}$ and $\mathcal{D}$ are both simply connected analytic spaces, it follows that $\phi$ is an isomorphism, as desired. 
\end{proof}

\subsection{Proof of Theorem \ref{Uniformization-thm}.}
The proof consists of three steps. First, we show that there is a map from the universal covering stack $\widetilde{\mathscr{X}-\mathscr{D}}$ of $\mathscr{X}-\mathscr{D}$ to the unit ball $\mathbb{B}^n$ which is locally biholomorphic. Next, we argue that this map is, in fact, an isomorphism, i.e., $\mathscr{X}-\mathscr{D}\cong[\mathbb{B}^d/\Gamma]$ for some lattice $\Gamma\in\mathrm{PU}(d,1)$. Finally, we show that $\mathscr{X}$ is birationally equivalent to compactification of $[\mathbb{B}^d/\Gamma]$ by a union of stacky quotients of abelian varieties by a finite group.

\subsection*{Step I} Let $(\mathcal{E},\theta)$ be the log Higgs bundle on $(\mathscr{X},\mathscr{D})$ given by $\mathcal{E}:=\Omega^1_{\mathscr{X}}(\log \mathscr{D})\oplus\mathcal{O}_\mathscr{X}$ with Higgs field
\begin{align*}
\theta:\Omega^1_{\mathscr{X}}(\log(\mathscr{D}))\oplus\mathcal{O}_{\mathscr{X}}&\to(\Omega^1_{\mathscr{X}}(\log(\mathscr{D}))\oplus\mathcal{O}_{\mathscr{X}})\otimes\Omega^1_{\mathscr{X}}(\log \mathscr{D})\\
(a,b)&\mapsto(0,1)\otimes a
\end{align*}
Let $f:Z_\bullet\to \mathscr{X}$ be a good covering for $(\mathscr{X},\mathscr{D})$. We want to show that the Bogomolov-Miyaoka-Yau inequality \ref{BMY1} holds.

By assumption, each log Higgs bundle $(f_k^*\mathcal{E},f_k^*\theta)$ on $(Z_k,f_k^*\mathscr{D})$ is slope-polystable with respect to $\mathcal{L}_k$. Then it follows from \cite[Theorem 6.5]{mochizuki_asterisque} that each $(f_k^*\mathcal{E},f_k^*\theta)$ satisfies the Bogomolov-Gieseker inequality
\begin{align*}\label{BMY2}
    (2\mathrm{rank}(f_k^*\mathcal{E})c_2(f_k^*\mathcal{E})-(\mathrm{rank}(f_k^*\mathcal{E})-1)c_1(f_k^*\mathcal{E})^2)\cdot c_1(\mathcal{L}_k)^{d-2}\ge0
\end{align*}
on $(Z_k,f_k^*\mathscr{D})$. Noting that $f_k^*\mathcal{E}=f_k^*\Omega^1_{\mathscr{X}}(\log \mathscr{D})\oplus\mathcal{O}_{Z_k}$, a straightforward computation yields
\begin{align*}
    (2(d+1)c_2(f_k^*\Omega^1_{\mathscr{X}}(\log \mathscr{D}))-dc_1(f_k^*\Omega^1_{\mathscr{X}}(\log \mathscr{D}))^2)\cdot c_1(\mathcal{L}_k)^{d-2}\ge0
\end{align*}
for all $k$. Since the collections $\{f_k^*\mathcal{E}\}$ and $\{\mathcal{L}_k\}$ satisfy compatibility, it follows that the inequality
\[
(2(d+1)c_2(f^*\Omega^1_{\mathscr{X}}(\log \mathscr{D}))-dc_1(f^*\Omega^1_{\mathscr{X}}(\log \mathscr{D}))^2)\cdot c_1(\mathcal{L})^{d-2}\ge0,
\] 
which is precisely \ref{BMY1}, holds on $Z_\bullet$. 

Note that the log Higgs bundle $(\mathcal{E},\theta)$ is actually a system of log Hodge bundles $\mathcal{E}=\mathcal{E}^{0,1}\oplus\mathcal{E}^{1,0}$ with $\mathcal{E}^{0,1}:=\mathcal{O}_\mathscr{X}$ and $\mathcal{E}^{1,0}:=\Omega^1_{\mathscr{X}}(\log \mathscr{D})$. It follows from Proposition \ref{Sys-of-Hodge-bundles-lem} that there is a principal system of log Hodge bundles $(P,\tau)$ on $(\mathscr{X},\mathscr{D})$ associated to $(\mathcal{E},\theta)$, whose structure group is $K=\mathrm{P}(\gl(d,\mathbb{C})\times\gl(1,\mathbb{C}))$. The Hodge group corresponding to $(P,\tau)$ is $G_0=\mathrm{PU}(d,1)$ and its maximal compact subgroup is $K_0=K\cap G_0=\mathrm{P}(\mathrm{U}(d)\times\mathrm{U}(1))$. The complexification of $G_0$ is the group $G=\mathrm{PGL}(\mathbb{C}^{d+1})$ and its adjoint representation $\mathrm{Ad}:G\to\mathrm{GL}(\mathfrak{g})$ is faithful. This is a Hodge representation by \cite[Example 3.9]{dengcadorel}, and by \cite[Example 3.10]{dengcadorel} it induces a system of log Hodge bundles $(P\times_{\mathrm{Ad}}\mathfrak{g},\mathrm{d}(\mathrm{Ad})(\tau))$ on $(\mathscr{X},\mathscr{D})$. It follows from the construction of $(P,\tau)$ that there is an isomorphism
\[
(P\times_{\mathrm{Ad}}\mathfrak{g},\mathrm{d}(\mathrm{Ad})(\tau))\cong(\mathcal{E}nd(\mathcal{E})^\perp,\theta_{\mathcal{E}nd(\mathcal{E})^\perp})
\]
of log Higgs bundles on $\mathscr{X}$, where $\mathcal{E}nd(\mathcal{E})^\perp$ denotes the trace-free subbundle of the endomorphism bundle $\mathcal{E}nd(\mathcal{E})$ and $\theta_{\mathcal{E}nd(\mathcal{E})^\perp}$ is the Higgs field induced from $(\mathcal{E},\theta)$. 
It is clear that $c_1(\mathcal{E}nd(\mathcal{E}))=0$, and since we assume that equality holds in \ref{BMY1}, a straightforward computation implies that 
\begin{align*}
ch_2(\mathcal{E}nd(\mathcal{E}))\cdot c_1(\mathcal{L})^{d-2}&=(-2\mathrm{rank}\mathcal{E}\cdot c_2(\mathcal{E})+(\mathrm{rank}\mathcal{E}-1)c_1(\mathcal{E})^2)\cdot c_1(\mathcal{L})^{d-2}\\
&=-(2(d+1)c_2(\Omega^1_{\mathscr{X}}(\log \mathscr{D}))-dc_1(\Omega^1_{\mathscr{X}}(\log \mathscr{D}))^2)\cdot c_1(\mathcal{L})^{d-2}\\
&=0
\end{align*}
holds on $\mathscr{X}$.

Since $(\mathcal{E},\theta)$ is polystable with respect to $\mathcal{L}$ by assumption, the same is true for $(\mathcal{E}nd(\mathcal{E}),\theta_{\mathcal{E}nd(\mathcal{E})})$. We now apply Proposition \ref{Hodge-metric} to obtain a Hodge metric $h$ for the system of log Hodge bundles $(\mathcal{E}nd(\mathcal{E})|_{\mathscr{X}-\mathscr{D}},\theta_{\mathcal{E}nd(\mathcal{E})}|_{\mathscr{X}-\mathscr{D}})$ on $\mathscr{X}-\mathscr{D}$ which is adapted to $(\mathcal{E}nd(\mathcal{E}),\theta_{\mathcal{E}nd(\mathcal{E})})$. There is a decomposition 
\[
(\mathcal{E}nd(\mathcal{E}),\theta_{\mathcal{E}nd(\mathcal{E})})=(\mathcal{E}nd(\mathcal{E})^\perp,\theta_{\mathcal{E}nd(\mathcal{E})^\perp})\oplus(\mathcal{O}_{\mathscr{X}},0)
\] 
as log Higgs bundles on $(\mathscr{X},\mathscr{D})$, from which we obtain a decomposition $h=h_1\oplus h_2$, where $h_1$ is the Hodge metric on the Higgs bundle $(\mathcal{E}nd(\mathcal{E})^\perp|_{\mathscr{X}-\mathscr{D}},\theta_{\mathcal{E}nd(\mathcal{E})^\perp}|_{\mathscr{X}-\mathscr{D}})$ which is adapted to $(\mathcal{E}nd(\mathcal{E})^\perp,\theta_{\mathcal{E}nd(\mathcal{E})^\perp})$, and $h_2$ is the canonical Hodge metric on the trivial log Higgs bundle $(\mathcal{O}_{\mathscr{X}},0)$. 

We apply Proposition \ref{Principal-VHS-prop} to obtain a metric reduction $P_H$ for $P|_{\mathscr{X}-\mathscr{D}}$ on $\mathscr{X}-\mathscr{D}$ with structure group $K_0=\mathrm{P}(\mathrm{U}(d)\times\mathrm{U}(1))\cong\mathrm{U}(d)$ induced by $h_1$, so that the triple $(P|_{\mathscr{X}-\mathscr{D}},\tau|_{\mathscr{X}-\mathscr{D}},P_H)$ is a principal variation of Hodge structures on $\mathscr{X}-\mathscr{D}$. In fact, since
\[
\tau:\mathcal{T}_{\mathscr{X}}(-\log \mathscr{D})\to P\times_K\mathfrak{g}^{-1,1}\cong\mathcal{H}om(\Omega^1_{\mathscr{X}}(\log \mathscr{D}),\mathcal{O}_{\mathscr{X}})
\]
is an isomorphism, $(P|_{\mathscr{X}-\mathscr{D}},\tau|_{\mathscr{X}-\mathscr{D}},P_H)$ is a uniformizing variation of Hodge structure over $\mathscr{X}-\mathscr{D}$.

It follows from Lemma \ref{Completeness-lem} that there is a morphism of analytic stacks
\begin{align}\label{Local-biholo}
\widetilde{\mathscr{X}-\mathscr{D}}\to G_0/K_0=\mathrm{PU}(d,1)/\mathrm{U}(d)=\mathbb{B}^d
\end{align}
which is \'etale. Since $\mathbb{B}^n$ is a smooth analytic space, it follows that the same holds for $\widetilde{\mathscr{X}-\mathscr{D}}$. This concludes the first step of the proof.

\subsection*{Step II} The next step is to prove that the morphism \ref{Local-biholo} is actually an isomorphism. Note that the metric reduction $P_H$ for $P|_{\mathscr{X}-\mathscr{D}}$ together with the metric $h_{\mathfrak{g}}$ defined in \cite[3.0.1]{dengcadorel} defines a natural metric $h_H$ on the bundle $P|_{\mathscr{X}-\mathscr{D}}\times_K\mathfrak{g}$. We want to show that the metric $h_H$ is complete on $\mathscr{X}-\mathscr{D}$. 

We assume from now on that the divisor $\mathscr{D}$ is smooth. This implies that $f^*\mathscr{D}$ is a smooth divisor on any good covering $f:Z_\bullet\to \mathscr{X}$ for $(\mathscr{X},\mathscr{D})$. Let $V_\bullet:=Z_\bullet-f^*\mathscr{D}$ and $V_k:=Z_k-f_k^*\mathscr{D}$ as usual. Consider the system of log Hodge bundles $(\mathcal{F},\nu):=(\mathcal{E}nd(\mathcal{E}),\theta_{\mathcal{E}nd(\mathcal{E})})$ on $(\mathscr{X},\mathscr{D})$. This pulls back to a system of log Hodge bundles $(f^*\mathcal{F},f^*\nu)$ on $(Z_\bullet,f^*\mathscr{D})$, which is the same as a collection of systems of log Hodge bundles $(f_k^*\mathcal{F},f_k^*\nu)$ on $(Z_k,f_k^*\mathscr{D})$ that satisfy compatibility. 

By the computations in Step II of the proof of \cite[Theorem 5.7(i)]{dengcadorel}, the Hodge metric $f_k^*h$ on $(f_k^*\mathcal{F}|_{V_k},f_k^*\nu|_{V_k})$ is complete, from which it follows that the metric $f_k^*h_H$ on $f_k^*P|_{V_k}\times_K\mathfrak{g}$ is complete for each $k$. Thus, the metric $f^*h_H$ is complete on $V_\bullet$ and descends to a complete metric, namely $h_H$, on $\mathscr{X}-\mathscr{D}$ because the covering $V_\bullet\to\mathscr{X}-\mathscr{D}$ is proper by Lemma \ref{hypercoveropen-lem}. 

We conclude from Lemma \ref{Completeness-lem} that $\widetilde{\mathscr{X}-\mathscr{D}}\cong\mathbb{B}^n$. This finishes the second step of the proof.

\subsection*{Step III} We begin by recalling the notion of a \emph{neat subgroup} of a connected semisimple algebraic group $\mathscr{G}$ from \cite[Chapter 17.1]{Borel69}.

\begin{definition}
    A subgroup $\Lambda\subset\mathscr{G}$ is said to be \emph{neat} if for all $g\in\Lambda$ and any faithful representation $\rho:\mathscr{G}\to\gl(n)$, the subgroup of $\mathbb{C}^*$ generated by the eigenvalues of $\rho(g)$ is torsion-free.
\end{definition}

It follows from \cite[Theorem 6.11]{Rag72} that any lattice ${PU}(d,1)$ admits a neat normal sublattice $\Gamma'$ of finite index. Let $\Gamma'\subset\Gamma$ be such a sublattice, set $U:=\mathbb{B}^d/\Gamma$ and $U':=\mathbb{B}^d/\Gamma'$ and note that $U'$ is a smooth quasiprojective variety. There is a finite surjective morphism $U'\to U$ and we have $U'\cong U/G$, where $G:=\Gamma/\Gamma'$.

Since $X$ is the moduli space of a smooth proper DM-stack, it has at worst quotient singularities. Suppose that $X$ is projective and let $X'$ denote the normalization of $X$ in the function field of $U'$ (see, e.g., \cite[Chapter XII, Section 9]{acgh2}). Then $X'$ is a normal projective variety, and there is a finite surjective morphism $X'\to X$ that extends the morphism $U'\to U$.

Let $Y'$ denote the (unique) toroidal compactification of $U'$ constructed in \cite[Chapter 5]{AMRT10} and note that the boundary $Y'-U'$ is a smooth divisor. Then by \cite[Lemma A.4]{dengcadorel} there is a birational morphism $X'\to Y'$. This is in fact an isomorphism by \cite[Lemma A.11]{dengcadorel}. In particular, there is a finite surjective morphism $Y'\to X$, which extends the morphism $U'\to U$. 

From the proof of \cite[Proposition A.1]{dengcadorel} it follows that the action of $G$ on $U'$ extends to $Y'$. We conclude from \cite[Lemmas A.13 and A.14]{dengcadorel} that $X\cong Y'/G$ and thus the boundary $D:=X-U$ is a disjoint union $\bigsqcup_iA_i/G$, where each $A_i$ is an abelian variety. 

The quotient $[Y'/G]$ is a smooth proper DM-stack that contains $[\mathbb{B}^d/\Gamma]$ as an open substack. The identity morphism of $[\mathbb{B}^n/\Gamma]$ then defines a birational map $\mathscr{X}\dashrightarrow[Y'/G]$ as claimed. This finishes the proof of Theorem \ref{Uniformization-thm}.
\qed\\
\\
Now we prove the converse Theorem \ref{Converse-thm} using the deep results in Section 6 and the appendix of \cite{dengcadorel}.

\subsection{Proof of Theorem \ref{Converse-thm}.}
Let $\Gamma'\subset\Gamma$ be a neat normal sublattice of finite index. The quotient map $\mathbb{B}^d\to[\mathbb{B}^d/\Gamma]$ factors as follows
\begin{center}
\begin{tikzcd}
\mathbb{B}^d\ar{r}\ar{dr}&{[\mathbb{B}^d/\Gamma']}\ar{r}\ar{d}&{[\mathbb{B}^d/\Gamma]}\ar{d} \\
&\mathbb{B}^d/\Gamma'\ar{r}&\mathbb{B}^d/\Gamma.
\end{tikzcd}
\end{center}
The vertical arrows are the natural maps to the moduli spaces. Note that the map $[\mathbb{B}^d/\Gamma']\to\mathbb{B}^d/\Gamma'$ is an isomorphism because $\Gamma'$ is torsion-free. We know from \cite[Theorem 5.2]{AMRT10} that there is a unique smooth toroidal compactification $X'$ of $[\mathbb{B}^d/\Gamma']$. Moreover, we deduce from the proof of \cite[Lemma A.2]{dengcadorel} that the natural action of the finite group $G:=\Gamma/\Gamma'$ on $\mathbb{B}^n/\Gamma$ extends to $X'$. 

We define $\mathscr{X}$ as the quotient stack $[X'/G]$. Hence, $\mathscr{X}$ is a smooth proper DM-stack that contains the ball quotient $[\mathbb{B}^d/\Gamma]$ as a dense open substack. By construction, the boundary $X'-\mathbb{B}^d/\Gamma'$ is a union of Abelian varieties, from which it follows that the boundary $\mathscr{D}=\mathscr{X}-[\mathbb{B}^d/\Gamma]$ is a union of quotients of Abelian varieties by the finite group $G$. This proves the first assertion of Theorem \ref{Converse-thm}.

Let $D':=X'-\mathbb{B}^d/\Gamma'$. It follows from \cite[Section 6.4]{dengcadorel} that the Chern class equality 
\begin{align*}
    2(d+1)c_2(\Omega^1_{X'}(\log D'))-dc_1(\Omega^1_{X'}(\log D'))^2=0
\end{align*}
holds on $X'$. Since $X'\to\mathscr{X}$ is a global \'etale chart, it follows that the equality 
\[
2(d+1)c_2(\Omega^1_{\mathscr{X}}(\log \mathscr{D}))-dc_1(\Omega^1_{\mathscr{X}}(\log \mathscr{D}))^2=0,
\]
which is precisely \ref{CCeq}, holds on $\mathscr{X}$. 

Consider the log Higgs bundle $(\mathcal{E}':=\Omega^1_{X'}(\log D')\oplus\mathcal{O}_{X'},\theta')$ on $(X',D')$, where $\theta'$ is the Higgs field defined in Theorem \ref{Uniformization-thm}. We know from \cite[Section 6.4]{dengcadorel} that the Higgs bundle $(\mathcal{E}'|_{X'-D'},\theta'|_{X'-D'})$ admits a Hermitian metric $h'$ that is 
\begin{enumerate}
\item Hermitian-Yang-Mills (\cite[Theorem 6.7]{dengcadorel})\label{one}
\item adapted to log order (\cite[Definition 5.2]{dengcadorel}), and \label{two}
\item acceptable (\cite[Definition 5.1]{dengcadorel}).\label{three}
\end{enumerate}
Note that $(\mathcal{E}',\theta')$ descends to the log Higgs bundle $(\mathcal{E}:=\Omega^1_{\mathscr{X}}(\log\mathscr{D})\oplus\mathcal{O}_\mathscr{X},\theta)$ on $\mathscr{X}$ and the metric $h'$ descends to a metric $h$ on $(\mathcal{E}|_{\mathscr{X}-\mathscr{D}},\theta|_{\mathscr{X}-\mathscr{D}})$ that satisfies properties \ref{one}--\ref{three} since $X'$ is a global \'etale chart for $\mathscr{X}$.

Let $f:Z_\bullet\to\mathscr{X}$ be a good covering for $(\mathscr{X,\mathscr{D}})$ and let $\mathcal{L}$ be an ample line bundle on $Z_\bullet$. Then the metric $f^*h$ on $(f^*\mathcal{E}|_{V_\bullet},\theta|_{V_\bullet})$ also satisfies properties \ref{one}--\ref{three}, where $V_\bullet:=Z_\bullet-f^*\mathscr{D}$. In particular, the metrics $f_k^*h$ on $(f_k^*\mathcal{E}|_{V_k},\theta|_{V_k})$ satisfy the three properties of \cite[Theorem 6.7]{dengcadorel} and it follows that $(f_k^*\mathcal{E},f_k^*\theta)$ is slope-polystable with respect to $\mathcal{L}_k$, for each $k$. Equivalently, $(\mathcal{E},\theta)$ is polystable with respect to $\mathcal{L}$. This concludes the proof of Theorem \ref{Converse-thm}. 
\qed

\appendix \section{Coskeletion and hypercovers.}
In this subsection, we recall relevant definitions from \cite{conradcohmdescent} which recalls and proves relevant statements on hypercovers.
\begin{definition}\label{defa1}
    Let $C$ be a category. Let $\operatorname{Simp}(C) = \operatorname{Fun}(\Delta^{\operatorname{op}},C)$ and for $n \ge 0$, let $\operatorname{Simp}_n(C) = \operatorname{Fun}(\Delta^{\operatorname{op}}_{\le n},C)$. We have similar notations for augmented simplicial objects $\operatorname{Simp}^{+}(C)$ and $\operatorname{Simp}^+_n(C)$.
\end{definition}

\begin{notation}
    The \textit{$n$-skeleton functor} $\operatorname{Simp}^+(C) \to \operatorname{Simp}^+_n(C)$ is the functor induced by the inclusion $\Delta_{\le n} \hookrightarrow \Delta$. We use the same notiation for the augmented simplicial objects too.
\end{notation}
\begin{example}
    Let $f: X \to Y$ be a map of schemes, then the \v{C}ech nerve of this morphism given by the diagram
    \begin{equation}
        \begin{tikzcd}
            \cdots X^{n+1}_Y \cdots X^2_Y \arrow[r, shift right=2] \arrow[r,shift right=-2]  & X \arrow[l,dotted] \arrow[r] & Y 
        \end{tikzcd}
    \end{equation}
    is an augmented simplicial object in the category of schemes. The diagram :
    \begin{equation}
          \begin{tikzcd}
   X^2_Y \arrow[r, shift right=2] \arrow[r,shift right=-2]  & X \arrow[l,dotted] \arrow[r] & Y 
        \end{tikzcd}
    \end{equation}

 is an example of an object in $\operatorname{Simp}^+_2(\operatorname{Sch})$.
 \end{example}
We recall the notion of coskeleton functor, which is defined as the right adjoint of the skeleton functor.
\begin{definition}\cite[Definition 3.2]{conradcohmdescent}
    For $n \ge 0$. an \textit{$n$-coskeleton functor} (if it exists) $\operatorname{cosk}_n : \operatorname{Simp}_n(C) \to \operatorname{Simp}(C)$ is a functor which is right adjoint to $\operatorname{sk}_n$.
\end{definition}

\begin{remark}
    When $C$ admits finite limits, the coskeleton functor is defined by the following formula:
    \begin{equation}
        \operatorname{cosk}_n(U)_m := \operatorname{lim}_{[p] \in (\Delta^{\le n}_{/[m]})} U([p]).
    \end{equation}
    where $\Delta^{\le n}_{/[m]}$ is the subcategory of $\Delta^{\le n}$ spanned by objects $[p]$ along with a morphism $[p] \to [m]$ (\cite[\href{https://stacks.math.columbia.edu/tag/0183}{Tag 0183}]{stacks-project}). For $n=0$, the $0$-coskeleta functor gives us the \v{C}ech nerve of a morphism. 
\end{remark}
The notion of coskeleton functors help us to introduce the notion of hypercovers.

\begin{definition}\cite[Definition 4.1]{conradcohmdescent}
    Let $C$ be a category admitting finite products and finite fiber products. Let $\mathbf{P}$ be a class of morphisms in $C$ which is stable under base change, preserved under compositions and contains all isomorphism. A simplicial object $X_{\bullet}$ in $C$ is said to be a $\mathbf{P}$-hypercovering if  for all $n \ge 0$, the natural adjunction map 
    \begin{equation}
        X_{\bullet} \to \operatorname{cosk}_n\operatorname{sk}_n(X_{\bullet})
    \end{equation}
    which in degree $n+1$ $X_{n+1} \to (\operatorname{cosk}_n\operatorname{sk}_n X)_{n+1}$ is an element in $\mathbf{P}$. If $X_{\bullet}$ is an augmented simplicial object, then we make a similar definition but also require it for case $n =-1$.
\end{definition}

We have also variants $\operatorname{sk}^n_m, \operatorname{cosk}^n_m$ between the categories $\operatorname{Simp}_n(C)$ and $\operatorname{Simp}_m(C)$ respectively.

In our case, we want to consider only specific hypercovers which are ``split". 
\begin{definition}
    Let $C$ be a category admitting finite coproducts. We say a simplicial object $X_{\bullet}$ in $C$ is \textit{split} if there exits subobjects $NX_j$ in each $X_j$ such that the natural map 
    \begin{equation}
        \coprod_{\phi : [n] \twoheadrightarrow [m]; m \le n}NX_m \to X_n
    \end{equation}
    is an isomorphism for every $n \ge 0$ where $NX_n \to X_m$ is induced by the map $NX_n \subset X_n \xrightarrow{X(\phi)}X_m$. Such a specification of such objects $NX_j$ is called a \textit{splitting} of $X_{\bullet}$.
\end{definition}
\begin{example}
    Let $f: X \to Y$ be a morphism in $C$ admitting a section $s: Y \to X$, then the \v{C}ech nerve of $f$ is a split-simplicial object of $C$. The decomposition of every term is produced by the section $s$. 
\end{example}
Before stating the key results of hypercovers, we also recall the notion of $m$-truncated $\mathbf{P}$-hypercoverings.
\begin{definition}\cite[Example 4.5]{conradcohmdescent}
    Let $X_{\bullet} \in \operatorname{Simp}_m(C)$. It is said to be  \textit{$m$-truncated $\mathbf{P}$-hypercovering} if for $n <m$, the unit map 
    \begin{equation}
        X_{\bullet} \to \operatorname{cosk}^n_m\operatorname{sk}^n_mX_{\bullet}
    \end{equation}
    is levelwise in $\mathbf{P}$.\end{definition}

\begin{remark}
  It turns out that, to construct split-simplicial hypercovers, one increases the truncation. The key argument in constructing split hypercovers using the notion of $m$-truncation is as follows :
    \begin{enumerate}
        \item Let $X_{\bullet}$ be an $m$-truncated split hypercovering. The goal is to construct an $m+1$-truncated hypercovering. 
        
        \item For this, we need 
        \begin{itemize}
            \item an object $N:=NX_m$ which we would like to "glue" at the $m+1$-skeleton of the $m+1$ truncated simplicial object 
            \item A map 
            \begin{equation}
                \beta : N \to (\operatorname{cosk}_mX)_{m+1}
            \end{equation}
        \end{itemize}
        \item The triple $(N,X,\beta)$ is sufficient enough to construct the $m+1$-truncated split hypercover $X'_{\bullet}$
        where roughly the objects are described as follows :
        \begin{equation}
            X'_l = \begin{cases}
                X_l \quad \quad \quad \quad \quad \quad \quad \quad \quad \quad \quad l \le m \\
                N \coprod \oplus_{[m+1] \twoheadrightarrow p, p \le m} X_p \quad \quad l= m+1
            \end{cases}
        \end{equation}
        The triple of data is collectively written as $\alpha(X)= (Y,N,\beta)$.
    \end{enumerate}
    The following exposition is more precise along with other properties of hypercoverings in the next proposition
\end{remark}

We now list the key properties of hypercoverings which are key for our purpose.
\begin{proposition}\label{propertyofPhypercoverings}
Let $C$ be category of smooth proper DM stacks (etale over base) and $\mathbf{P}$ be the class of proper (etale) surjective morphisms. Then:
    \begin{itemize}
    \item For any split-simplicial $m$-truncated object $Y$, an object $N$ in $C$ and a morphism 
    \begin{equation}
        \beta : N \to (\operatorname{coskn}_m Y)_{m+1},
    \end{equation}
    there exists a split-$m+1$-truncated simplicial object $X$ such that $\alpha(X)=(N,Y,\beta)$ and it is unique upto isomorphism. (\cite[Theorem 4.12]{conradcohmdescent}).        
    \item If $X_{\bullet}$ is a $m$-truncated $\mathbf{P}$-hypercovering, then $\operatorname{coskn}_m(Y_{\bullet})$ is a $\mathbf{P}$-hypercovering (\cite[Example 4.5]{conradcohmdescent}).
        \item Every simplicial $\mathbf{P}$-hypercovering $X_{\bullet}$ of $C$  with the property that $\operatorname{sk}_n(X_{\bullet})$ is split admits a split refinement i.e. there exists another simplicial object $X'_{\bullet}$ of $C$ with a map $f:X'_{\bullet} \to X_{\bullet}$ such that $\operatorname{sk}_n(f)$ is an isomorphism and $X'_{\bullet}$ is split (\cite[Thoerem 4.13]{conradcohmdescent}).
\item Every $m$-truncated $\mathbf{P}$-hypercovering has proper (\'etale) face and degeneracy maps (\cite[Corollary 4.14]{conradcohmdescent}).
 
         \end{itemize}
\end{proposition}

\bibliographystyle{plain}
\bibliography{sources}

@Article{Simpson_unif,
 Author = {Simpson, Carlos T.},
 Title = {Constructing variations of {Hodge} structure using {Yang}-{Mills} theory and applications to uniformization},
 FJournal = {Journal of the American Mathematical Society},
 Journal = {J. Am. Math. Soc.},
 ISSN = {0894-0347},
 Volume = {1},
 Number = {4},
 Pages = {867--918},
 Year = {1988},
 Language = {English},
 DOI = {10.2307/1990994},
 zbMATH = {4096412},
 Zbl = {0669.58008}
}

@Article{dengcadorel,
 Author = {Deng, Ya and Cadorel, Benoît},
 Title = {A characterization of complex quasi-projective manifolds uniformized by unit balls},
 FJournal = {Mathematische Annalen},
 Journal = {Math. Ann.},
 ISSN = {0025-5831},
 Volume = {384},
 Number = {3-4},
 Pages = {1833--1881},
 Year = {2022},
 Language = {English},
 DOI = {10.1007/s00208-021-02334-z},
 zbMATH = {7605983},
 Zbl = {1498.14022}
}

@article {simpsonstacks,
    AUTHOR = {Simpson, Carlos},
     TITLE = {Local systems on proper algebraic {$V$}-manifolds},
   JOURNAL = {Pure Appl. Math. Q.},
  FJOURNAL = {Pure and Applied Mathematics Quarterly},
    VOLUME = {7},
      YEAR = {2011},
    NUMBER = {4},
     PAGES = {1675--1759},
      ISSN = {1558-8599,1558-8602},
   MRCLASS = {14D23 (14A20 14C30 32C18)},
  MRNUMBER = {2918179},
MRREVIEWER = {Jack\ Hall},
       DOI = {10.4310/PAMQ.2011.v7.n4.a27},
       URL = {https://doi.org/10.4310/PAMQ.2011.v7.n4.a27},
}

@article {mochizuki_asterisque,
    AUTHOR = {Mochizuki, Takuro},
     TITLE = {Kobayashi-{H}itchin correspondence for tame harmonic bundles
              and an application},
   JOURNAL = {Ast\'{e}risque},
  FJOURNAL = {Ast\'{e}risque},
    NUMBER = {309},
      YEAR = {2006},
     PAGES = {viii+117},
      ISSN = {0303-1179,2492-5926},
      ISBN = {978-2-85629-226-6},
   MRCLASS = {32Q20 (14D07 14J60 32G20 53C07)},
  MRNUMBER = {2310103},
MRREVIEWER = {Julien\ Keller},
}

@article {GKPT,
    AUTHOR = {Greb, Daniel and Kebekus, Stefan and Peternell, Thomas and
              Taji, Behrouz},
     TITLE = {The {M}iyaoka-{Y}au inequality and uniformisation of canonical
              models},
   JOURNAL = {Ann. Sci. \'{E}c. Norm. Sup\'{e}r. (4)},
  FJOURNAL = {Annales Scientifiques de l'\'{E}cole Normale Sup\'{e}rieure. Quatri\`eme
              S\'{e}rie},
    VOLUME = {52},
      YEAR = {2019},
    NUMBER = {6},
     PAGES = {1487--1535},
      ISSN = {0012-9593},
   MRCLASS = {14E30 (32Q30)},
  MRNUMBER = {4061021},
MRREVIEWER = {James McKernan},
       DOI = {10.24033/asens.2414},
       URL = {https://doi.org/10.24033/asens.2414},
}

@book {acgh2,
    AUTHOR = {Arbarello, E. and Cornalba, M. and Griffiths, P.},
     TITLE = {Geometry of algebraic curves. {V}olume {II}},
    SERIES = {Grundlehren der Mathematischen Wissenschaften},
    VOLUME = {268},
      NOTE = {With a contribution by J. Harris},
 PUBLISHER = {Springer},
   ADDRESS = {Heidelberg},
      YEAR = {2011},
     PAGES = {xxx+963},
      ISBN = {978-3-540-42688-2},
   MRCLASS = {14H10},
  MRNUMBER = {2807457},
}

@misc{conradcohmdescent,
author ={{C}onrad, {B}rian},
title = {{C}ohomological {D}escent},
howpublished={\url{https://math.stanford.edu/~conrad/papers/hypercover.pdf}},
}

@misc{GP24,
    author = {Graf, Patrick and Patel, Aryaman},
    title = {Uniformization of klt pairs by bounded symmetric domains},
    Eprint = {2410.12753},
    Eprinttype = {arxiv},
    year = {2024}
}

@book{Borel69,
    author = {Borel, Armand},
    title = {Introduction aux groupes arithmétiques, Actualités scientifiques et industrielles},
    publisher = {Hermann, Paris},
    year = {1969}
}

@book{AMRT10,
    author = {Ash, Avner and Mumford, David and Rappoport, Michael and Tai, Yung-sheng},
    title = {Smooth Compactifications of Locally Symmetric Varieties. With the Collaboration of Peter Scholze, 2nd edition},
    publisher = {Cambridge University Press, Cambridge},
    year = {2010}
}

@book{Rag72,
    author = {Raghunathan, M.S.},
    title = {Discrete Subgroups of Lie Groups, Results and Problems in Cell Differentiation},
    publisher = {Springer, Berlin},
    year = {1972}
}

@misc{stacks-project,
  author       = {The {Stacks project authors}},
  title        = {The Stacks project},
  howpublished = {\url{https://stacks.math.columbia.edu}},
  year         = {2026},
}

@article{P23,
    author = {Patel, Aryaman},
    title = {Uniformization of projective klt varieties by bounded symmetric domains},
    journal = {Selecta Math. (New Series)},
    year = {2026},
    volume = {32},
    number = {65}
}

\end{document}